\documentclass[11pt]{amsart}
\usepackage[utf8x]{inputenc}
\usepackage[english]{babel}
\usepackage{bbm}
\usepackage{mathtools}
\usepackage{thmtools}
\usepackage{fourier}
\usepackage[dvipsnames]{xcolor}
\declaretheorem[shaded={rulecolor=Lavender,
  rulewidth=1pt},within=section,name=Theorem]{thm}
\declaretheorem[shaded={rulecolor=Lavender,
  rulewidth=1pt},numbered=no,name=Theorem]{thm*}
\declaretheorem[shaded={rulecolor=Lavender, rulewidth=1pt},name=Theorem]{bigthm}

\declaretheorem[shaded={rulecolor=Lavender, rulewidth=1pt},
sibling=thm,name=Lemma]{lem}
\declaretheorem[shaded={rulecolor=Lavender, rulewidth=1pt},
numbered=no,name=Lemma]{lem*}
\declaretheorem[shaded={rulecolor=Lavender, rulewidth=1pt},
sibling=thm,name=Corollary]{cor}
\declaretheorem[style=definition, shaded={rulecolor=Lavender,
  rulewidth=1pt}, sibling=thm,name=Definition]{defn}
\declaretheorem[style=definition, shaded={rulecolor=Lavender,
  rulewidth=1pt}, sibling=thm,name=Unsolved problem]{uproblem}
\declaretheorem[style=definition, shaded={rulecolor=Lavender,
  rulewidth=1pt}, sibling=thm,name=Conjecture]{conj}
\declaretheorem[style=remark, shaded={rulecolor=Lavender,
  rulewidth=1pt}, sibling=thm,name=Remark]{rem}
\declaretheorem[style=remark, shaded={rulecolor=Lavender,
  rulewidth=1pt}, sibling=thm,name=WARNING]{warn}

\numberwithin{equation}{section}

\usepackage[hypertexnames=false,hyperindex,breaklinks,bookmarks,pagebackref]{hyperref}
\usepackage{cleveref}
\crefname{lem}{lemma}{lemmas}
\Crefname{lem}{Lemma}{Lemmas}
\crefname{bigthm}{theorem}{theorems}
\Crefname{bigthm}{Theorem}{Theorems}
\crefname{conj}{conjecture}{conjectures}
\Crefname{conj}{Conjecture}{Conjectures}
\crefname{bigconj}{conjecture}{conjectures}
\Crefname{bigconj}{Conjecture}{Conjectures}
\crefname{rem}{remark}{remarks}
\Crefname{rem}{Remark}{Remarks}
\crefname{thm}{theorem}{theorems}
\Crefname{thm}{Theorem}{Theorems}
\crefname{cor}{corollary}{corollaries}
\Crefname{cor}{Corollary}{Corollaries}
\crefname{defn}{definition}{definitions}
\Crefname{defn}{Definition}{Definitions}
\crefname{problem}{problem}{problems}
\Crefname{problem}{Problem}{Problems}

\newcommand{\seq}[2]{\ensuremath{#1_1, #1_2,\dots,#1_{#2},\dots}}
\DeclarePairedDelimiter\paren{\lparen}{\rparen}
\DeclarePairedDelimiter\sbrack{\lbrack}{\rbrack}
\DeclarePairedDelimiter\abs{\lvert}{\rvert}
\DeclarePairedDelimiter\norm{\|}{\|}
\DeclarePairedDelimiter\cbrace{\lbrace}{\rbrace}

\usepackage{anysize}
\marginsize{2cm}{2cm}{2cm}{2cm}
\usepackage{autonum}
\begin{document}
\bibliographystyle{halpha} \title{Random differences in Szemer\'edi's
  theorem and related results}

% \author{\greektext Nikolaos Frantzhkinakhs} \address{\greektext
% Τμήμα Μαθηματικών του Πανεπιστημίου Κρήτης} \author{Emmanuel
% Lesigne} \address{\foreignlanguage{french}{Département de gros
% fromages, Université François Rabelais}}

\author{Nikos Frantzikinakis} \address{Department of Mathematics,
  University of Crete, Heraklion, Greece} \author{Emmanuel Lesigne}
\address{Université François Rabelais, CNRS, LMPT UMR 7350, Tours,
  France} \author{M\'at\'e Wierdl} \address{Department of Mathematical
  Sciences, The University of Memphis, Memphis, TN, USA} \thanks{The
  first author has been partially supported by Marie Curie IRG 248008}\thanks{The
  third author has been partially supported by the US National Science
  Foundation, award number DMS-1102634}

\begin{abstract}
  We introduce a new, elementary method for studying random
  differences in arithmetic progressions and convergence phenomena
  along random sequences of integers. We apply our method to obtain
  significant improvements on two results. The first improvement is
  the following: Let $\ell$ be a positive integer and let $\mathbbm
  u_1\ge \mathbbm u_2\ge \dots$ be a decreasing sequence of
  probabilities satisfying $\mathbbm u_n\cdot n^{1/(\ell +1)}\to
  \infty$. Let $R=R^\omega$ be the random sequence we get by selecting
  the natural number $n$ with probability $u_n$. If $A$ is a set of
  natural numbers with positive upper density, then $A$ contains an
  arithmetic progression $a, a+r,a+2r,\dots,a+\ell r$ of length $\ell
  +1$ with difference $r\in R^\omega$. The best earlier result (by
  M. Christ and us) was the condition $\mathbbm u_n\cdot
  n^{2^{-\ell+1}}\to \infty$ with a logarithmic speed.  The new bound
  is better when $\ell \ge 4$.

  The other improvement concerns almost everywhere convergence of
  double ergodic averages: we (randomly) construct a sequence
  $r_1<r_2<\dots$ of positive integers so that for any $\epsilon>0$ we
  have $r_n/n^{2-\epsilon}\to \infty$ and for any measure preserving
  transformation $T$ of a probability space the averages
  \begin{equation}
    \frac1N\sum_{n<N}T^nF_1(x)T^{r_n}F_2(x)
  \end{equation}
  converge for almost every $x$. Our earlier best result was the
  $r_n/n^{(1+1/14)-\epsilon}\to \infty$ growth rate on the sequence
  $(r_n)$.
\end{abstract}

\maketitle

\dedicatory{We dedicate this paper to the 70th birthday of our friend,
  Karl Petersen, who raised the interest of the senior authors in
  random sequences 20 years ago.}
\tableofcontents
% \listoftheorems

\section{Conventions}
\label{sec:conventions}

In this paper, a \emph{dynamical system} is a quadruple
$(X,\Sigma,\mu,T)$ where $(X,\Sigma,\mu)$ is a probability space, and
$T$ is a $\Sigma$-measurable, measure preserving transformation,
i.e. for any $A\in \Sigma$, we have $\mu(T^{-1}A)=\mu(A)$. Since the
$\sigma$-algebra $\Sigma$ doesn't play any special role in what
follows, we omit it from our notation, and we just write $(X,\mu,T)$
instead of $(X,\Sigma,\mu,T)$.

When we talk about functions defined on some measure space, we always
assume, the function is measurable.

We use the usual notational convention that the isometry of $L^1$
induced by a measure preserving transformation $T$ is also denoted by
the same letter $T$: $Tf(x)=f(Tx)$.

The arbitrary constants $c$ or $C$ appearing in an estimate are always
positive, and are absolute unless otherwise noted.  The same constant
may have a different value even in the same series of estimates.

The notation $\sum_{n<N}$ means $\sum_{0\le n<N}$, so $0$ is included
in the range of summation.

A \emph{lacunary sequence} of positive real numbers or integers is a
sequence $\cbrace*{a_1<a_2<a_3<\dots <a_n<\dots}$ such that
$\liminf_{n\to\infty} \frac{a_{n+1}}{a_n}>1$.

The \emph{cardinality} of a set $A$ is denoted by $|A|$.

\section{ Introduction, History}
\label{sec:introduction-history}

\subsection{Intersective and recurrent sets}
\label{sec:inters-recurr-sets}

For a set $A$ of positive integers, let $A(M)$ denote its\emph{ counting
function}, the number of elements in $A$ that are less than $M$,
\begin{equation}
  \label{eq:1}
  A(M)=\abs[\Big]{\cbrace*{a\mid a\in A, \,a<M}}.
\end{equation}
We denote by $\overline d(A)$ the \emph{upper density} of $A$
\begin{equation}
  \label{eq:2}
  \overline d(A)=\limsup_{M\to\infty} \frac{A(M)}{M}.
\end{equation}
The celebrated result of \cite{MR0422191,szem2} says that if $A$ is a
set of positive integers with positive upper density, then for any
positive $\ell$ one can find an $\ell+1$ long arithmetic progression
$a, a+r,a+2r,\dots,a+\ell r$ in $A$. One way to generalize this result
is to demand that the difference $r$ is from some prescribed sequence.
\begin{defn}
  Let $\ell$ be a positive integer.  We say that a set $R$ of positive
  integers is \emph{$\ell$-recurrent} or \emph{$\ell$-intersective} if
  any set $A$ of positive upper density contains an $\ell+1$ long arithmetic
  progression $a, a+r,a+2r,\dots,a+\ell r$ with the difference $r\in
  R$.

  We refer to $\ell$ as the \emph{order} of intersectivity.
\end{defn}
We note that an $\ell$-intersective set may not be
$(\ell+1)$-intersective, see \cite{MR2266880}.

The explanation for the terminology ``intersective'' is that saying
$R$ is $\ell$-intersective is equivalent with
\begin{equation}\label{eq:3}
   \overline d(A)>0\text{ implies } A\cap(A-r)\cap(A-2r)\cap\dots\cap (A-\ell r)\ne \emptyset \text{ for some }
  r\in R. 
\end{equation}
For the term ``recurrent'', it is known that \cref{eq:3} is equivalent
with the following: In every dynamical system $(X,\mu,T)$
\begin{equation}
  \label{eq:4}
  \mu(A)>0\text{ implies } \mu\paren*{ A\cap T^{-r}A\cap
    T^{-2r}A\cap\dots\cap T^{-\ell r}A}>0\text{ for some }
  r\in R.
\end{equation}
The direction ``recurrent $\implies$ intersective'' is known as
Furstenberg's transfer principle, and it was proved in
\cite{MR0498471}.

It is known that the following sets are $\ell$-intersective
\begin{itemize}
\item the squares $\cbrace{1^2,2^2,3^2\dots,n^2,\dots}$ or, more
  generally, the $k$th powers
  $\cbrace{1^k,2^k,3^k\dots,n^k,\dots}$. For $\ell =1$, this was a
  conjecture of Lovász, and was answered in
  \cite{MR0466059,MR487031}. The case of any $\ell$ was proved in
  \cite{MR1325795}.
\item the ``prime minus ones'' $\cbrace{2-1, 3-1,
    5-1,\dots,p-1,\dots}$, or the ``prime plus ones'' $\cbrace{2+1,
    3+1, 5+1,\dots,p+1,\dots}$ where $p$ is a prime number. For
  $\ell=1$ this was a question raised by Erdős, and was answered by
  \cite{MR487031}.  The $\ell=2$ case is solved in \cite{MR2361086},
  the case of all other $\ell$ is solved in \cite{MR2999293}.
\item
  $\cbrace*{\sbrack{1^{3/2}},\sbrack{2^{3/2}},\sbrack{3^{3/2}},\dots,\sbrack{n^{3/2}},\dots}$
  or, more generally,
  $\cbrace*{\sbrack{1^{\delta}},\sbrack{2^{\delta}},\sbrack{3^{\delta}},\dots,\sbrack{n^{\delta}},\dots}$
  for any fixed positive, non-integer $\delta$. For $\ell=1$, this
  probably was first noted in \cite{MR2638457}, but it's clear that
  Boshernitzan had been aware of the result. See \cite{MR725629}. It
  also follows from the general result in \cite{MR515054}. For all
  $\ell$, it is proved in \cite{MR2531366}.
\end{itemize}
These intersective sequences don't increase faster to infinity than
some power of $n$. In terms of the counting function $R(N)$ this
means that any of the above sets is not very thin since it satisfies
$R(N)>N^{\epsilon}$ for some positive $\epsilon$. Are there any
intersective sets which increase to infinity faster than any power of
$n$?  The question is not exactly a good one: \cite{MR515054}
considered sets with dilations. A $k$-long dilation is a set of
integers of the form $\cbrace*{d,2d,3d,\dots,kd}$. Now, if a set
contains arbitrarily long dilations, then it is a consequence of
Szemerédi's theorem, that the set is $\ell$-intersective for all
$\ell$.  This way, one can make $\ell$-intersective sequences with
arbitrarily thin counting functions.  On the other hand, a lacunary
sequence cannot contain arbitrarily long dilations and it is shown in
\cite{esII} that a lacunary sequence is not even $1$-intersective.

Are then the previously mentioned intersective sets exceptional in
some sense, or a typical set is also intersective?  To be able to talk
more conveniently about these questions, we introduce the concept of a
\emph{random sequence}:
\begin{defn}
  Given a sequence $\mathbbm u_1,\mathbbm u_2,\dots$ of probabilities,
  let $U_1,U_2,\dots$ be a sequence of independent $0-1$ valued
  random variables on a probability space $(\Omega,P)$ so that
  $P(U_n=1)=\mathbbm u_n=1-P(U_n=0)$.  We think of the sequence
  $U_1(\omega),U_2(\omega),\dots$ as the indicator of the set
  $R^\omega$ defined by
  \begin{equation}
    \label{eq:5}
    R^\omega=\cbrace*{n\mid U_n(\omega)=1}.
  \end{equation}
  We say that \emph{the random sequence has some given property} if
  there is a set $\Omega'\subset \Omega$ with $P(\Omega')=1$ so that
  $R^\omega$ has the property for all $\omega\in\Omega'$.  Note that
  the concept of a random sequence is always tied to a given sequence
  of probabilities $\mathbbm u_1,\mathbbm u_2,\dots$.
\end{defn}
\begin{warn}
  Unless we say otherwise, in this paper we assume that the
  probabilities $\mathbbm u_n$ form a \emph{non-increasing} sequence,
  $\cbrace*{\mathbbm u_1\ge \mathbbm u_2\ge \mathbbm u_3\ge \dots}$,
  and that they sum to $\infty$, $\sum_n \mathbbm u_n=\infty$.
\end{warn}
The standing assumption $\sum_n \mathbbm u_n=\infty$ guarantees, by
the strong law of large numbers (cf. \Cref{thm:2} in the Appendix),
that the random sequence $R$ is infinite, since its counting function
$R(N)=\sum_{n<N}U_n$ is almost surely asymptotically equal to
$\sum_{n<N} \mathbbm u_n$. We make the decreasing assumption because
intuitively it gives the expected behavior from a random set.

Now we can make the concept of a typical set completely precise: it is
the random set associated with a given sequence of probabilities
$\mathbbm u_n$. It is easy to see that if $\mathbbm u_n$ goes to $0$
slower than any negative power of $n$, that is $\mathbbm u_n\cdot
n^\epsilon\to \infty$ for any $\epsilon>0$, then the random sequence
contains arbitrarily long dilations (and intervals!), hence it is
$\ell$-intersective for all $\ell$.  On the other hand, if $\mathbbm
u_n < n^{-\epsilon}$ for some positive $\epsilon$, then the random
sequence doesn't have arbitrarily long dilations. Nevertheless,
\cite{MR515054} showed that if $\mathbbm u_n > \log n \cdot n^{-1/3}$,
then the random sequence is $1$-intersective. Can a random sequence be
arbitrarily thin but still intersective? Well, it is shown in
\cite{MR725629} that if $\mathbbm u_n\cdot n\to \infty$, then the
random sequence is still 1-intersective.  Paul Balister pointed out to
us that if $\mathbbm u_n=\frac{1}{n(\log n)^{1/2+\epsilon}}$, then the
random sequence is lacunary and hence is not 1-intersective. Anthony
Quas told us in private conversation that the method of
\cite{MR920079} can be used to show that if $\mathbbm u_n=1/n$, then
the random sequence is not $1$-intersective.

While quite a lot is known about 1-intersective sets, much less is
known about typical $2$ or higher order intersective sets.  We proved
in \cite{Fra} that if $\mathbbm u_n=n^{-b}$ for some some $b<1/2$,
then the random set is $2$-intersective.

For general $\ell$, we proved in \cite{Fra} that if
\begin{equation}
  \label{eq:6}
  \mathbbm u_n=n^{-b} \text{ for some } 0<b<\frac1{2^{\ell -1}},
\end{equation}
then the random set is $\ell$-intersective. This result was proved
independently from us in \cite{christ:random}.

The main result on random intersective sets in this current paper is
the following improvement of the above results
\begin{bigthm}[Random $\ell$-intersectivity theorem]
  \label{bigthm:1}
  Let $\ell$ be a positive integer, and assume the decreasing sequence
  $\mathbbm u_1\ge \mathbbm u_2\ge \dots$ of probabilities satisfy
  \begin{equation}
    \label{eq:7}
    \mathbbm u_n\cdot n^{1/(\ell +1)}\to \infty.
  \end{equation}
   
  Then the random sequence is $\ell$-intersective.
\end{bigthm}
The condition in \cref{eq:7} is better than the condition in
\cref{eq:6} for $\ell \ge 4$. The earlier result that used the
condition in \cref{eq:6} was obtained by using Van der Corput's
inequality $\ell -1$ times, and this iteration poses the restriction
in \cref{eq:6}.  Our method doesn't use iteration, and it doesn't
reduce the case for $\ell$ to the case for $\ell -1$. In particular,
we give a new, quite elementary proof when $\ell =1$. If $\ell =1$,
\cref{eq:7} gives the restriction $b<1/2$, while we know, after
Boshernitzan, that every $b<1$ works.  Our proof doesn't use Fourier
analysis or the spectral theorem.
\begin{rem}
  \label{rem:1}
  We will see that instead of the individual behavior of the $\mathbbm
  u$ probabilities, as expressed in \cref{eq:7}, only their average
  behavior is important.  The exact assumption we use in the proof is
  \begin{equation}
    \label{eq:8}  
    \frac{\sum_{n<N}\mathbbm u_n}{N^{1-1/(\ell +1)}}\to \infty.
  \end{equation}
  In other words, only the counting function $R(N)$ of the
  random set $R$ matters: by the strong law of large numbers, the
  above is the same as saying that $R(N)/N^{1-1/(\ell +1)}\to\infty$.

  This indicates that, in fact, the proof applies equally well to a
  finitary version of \Cref{bigthm:1}; see
  \Cref{sec:finit-vers-crefb}.
\end{rem}

Now that we are close to the end of this introductory section, we state a
conjecture which, if true, would be an extension of Boshernitzan's
result for random $1$-intersective sets to $\ell$-intersective ones
\begin{conj}\label{sec:introduction-history-1}
  Suppose
  \begin{equation}
    \label{eq:9}
    \mathbbm u_n\cdot n \to\infty.
  \end{equation}
  
  Then the random sequence is $\ell$-intersective.
\end{conj}
As we mentioned earlier, if this conjecture is true, it's sharp in the
sense that if $\mathbbm u_n=1/n$, then the random sequence is not
$1$-intersective.

\subsection{Convergence of ergodic averages}
\label{sec:conv-ergod-aver}

In this section $(X,\mu,T)$ denotes a dynamical system.
\begin{defn}
  Let $\ell$ be a positive integer.  We say that the set
  $R=\cbrace*{r_1<r_2<r_3<\dots}$ of positive integers is
  \emph{$\ell$-averaging in norm} if in any dynamical system the
  averages
  \begin{equation}
    \label{eq:10}
    \frac1{N}\sum_{n<N}T^{r_n}F_1\cdot T^{2r_n}F_2\cdots T^{\ell r_n}F_\ell
  \end{equation}
  converge in $L^1$ norm for any choice of bounded functions $F_i$ on
  $X$.
\end{defn}
While $1$-averaging sequences have been studied for a long time, the
study of $\ell$-averaging sequences for $\ell>1$ is more recent than
that of $\ell$-intersective sets: That the natural numbers are
$2$-averaging was proved in \cite{MR0498471}, the $\ell=3$ case was
proved in \cite{MR788966}, \cite{MR1827115}, and the breakthrough result of
\cite{MR2150389}, where they show that the set of natural numbers is
$\ell$-averaging in norm, came 30 years after Szemerédi's and
Furstenberg's result. Now the study of intersective and averaging
sequences have been done parallel to each other: The following sets
have been found $\ell$-averaging in norm
\begin{itemize}
\item the squares $\cbrace{1^2,2^2,3^2,\dots,n^2,\dots}$ or, more
  generally, $\cbrace{p(1),p(2),p(3),\dots,p(n),\dots}$, where $p(x)$
  is polynomial taking on positive integer values for large enough
  $n$. For $\ell = 1$, this is a folklore result, and follows from
  Weyl's estimate on trigonometric sums. For all $\ell$, it was proved
  in \cite{MR2191208} (weak convergence) and in \cite{MR2151605} for
  norm convergence.
\item $\cbrace{2, 3, 5,\dots,p,\dots}$, where $p$ is a prime
  number. For $\ell=1$ this follows from Vinogradov's classical
  estimates on the Fourier transform (trigonometric sum) of the
  sequence of primes.  The $\ell=2$ case is solved in
  \cite{MR2361086}, the case of all other $\ell$ is solved in
  \cite{MR2999293}.
\item
  $\cbrace*{\sbrack{1^{3/2}},\sbrack{2^{3/2}},\sbrack{3^{3/2}},\dots,\sbrack{n^{3/2}},\dots}$
  or, more generally,
  $\cbrace*{\sbrack{1^{\delta}},\sbrack{2^{\delta}},\sbrack{3^{\delta}},\dots,\sbrack{n^{\delta}},\dots}$
  for any fixed positive, non-integer $\delta$. For $\ell=1$, this may
  have been first noted explicitly only in \cite{MR2638457}, but it's
  clear that Boshernitzan had been aware of the result. See
  \cite{MR725629}. For all $\ell$, the result is proved in
  \cite{MR2762998}.
\end{itemize}
An important difference between intersective and averaging sets is
that just because a set contains arbitrarily long dilations, it may
not be an averaging set. In particular, it is not easy to construct
fast growing averaging sequences. To avoid trivial examples, such as
sequences consisting of longer and longer intervals of consecutive
integers, we are asking the following: are there any averaging sets
$\cbrace*{r_1,r_2<r_3<\dots}$ so that the gaps $r_{n+1}-r_n$ between
consecutive terms goes to infinity faster than any polynomial? It
follows from the estimates in \cite{MR0330068} (see also \cite[Theorem
2]{MR1625541}), that if, for some positive $\epsilon$,
$r_n=\sbrack*{\exp{\paren*{ (\log n)^{3/2-\epsilon}}}}$, then
$\cbrace*{ r_1<r_2<r_3<\dots}$ is $1$-averaging. It is not known if
this sequence is $2$-averaging (or, for that matter, if it's
$2$-intersective). How fast can an averaging sequence increase? Can it
increase arbitrarily fast? It was \cite{MR725629} (and, explicitly,
\cite{MR937581}) who showed that if
\begin{equation}
  \label{eq:11}
  \mathbbm u_n\cdot n\to\infty,
\end{equation}
then the random sequence is $1$-averaging. As in case of
$1$-intersective sets, it is known that this result is best possible:
\cite{MR1721622} proved that if
\begin{equation}
  \label{eq:12}
  \mathbbm u_n=\frac1n,
\end{equation}
then the random sequence is \emph{not} $1$-averaging.  How about
$2$-averaging random sets? Well, we still don't know of any
$2$-averaging random set $\cbrace*{r_1,r_2<r_3<\dots}$ which goes to
infinity as fast as the squares, that is when $\mathbbm u_n=n^{-1/2}$.

For all $\ell$, we proved in \cite{Fra} that if
\begin{equation}
  \label{eq:13}
  \mathbbm u_n=n^{-b} \text{ for some } 0<b<\frac1{2^{\ell -1}},
\end{equation}
then the random set is $\ell$-averaging. This result was also proved
independently from us by \cite{christ:random}. Our main result on
random $\ell$-averaging sets in the present paper is the following
\begin{bigthm}[Random $\ell$-averaging theorem]
  \label{bigthm:2}
  Let $\ell$ be a positive integer, and suppose the decreasing
  sequence $\mathbbm u_1\ge\mathbbm u_2\ge\dots$ of probabilities
  satisfies
  \begin{equation}
    \label{eq:14}
    \mathbbm u_n\cdot n^{1/(\ell +1)}\to \infty.
  \end{equation}

  Then the random sequence $R^\omega=\cbrace{r_1<r_2<r_3<\dots}$ is
  $\ell$-averaging in norm, and the limit of the averages
  \begin{equation}
    \label{eq:15}
    \frac1{N}\sum_{n<N}T^{r_n}F_1\cdot T^{2r_n}F_2\cdots T^{\ell r_n}F_\ell
  \end{equation}
  along the sequence is the same as the limit of the multiple averages
  \begin{equation}
    \label{eq:16}
    \frac1N \sum_{n<N}T^{n}F_1\cdot T^{2n}F_2\cdots T^{\ell n}F_\ell.
  \end{equation}

\end{bigthm}
Here, again, the real assumption the proof uses is in terms of the
average behavior of the the $\mathbbm u_n$, that is, in terms of the
counting function of the random sequence: the assumption in
\cref{eq:14} can be weakened to
\begin{equation}
  \label{eq:17}
  \frac{\sum_{n<N}\mathbbm u_n}{N^{1-1/(\ell +1)}}\to \infty.
\end{equation}

We have a conjecture similar to the one for intersective sets
\begin{conj}\label{sec:conv-ergod-aver-1}
  Suppose
  \begin{equation}
    \label{eq:18}
    \mathbbm u_n\cdot n\to\infty.
  \end{equation}
  
  Then the random sequence is $\ell$-averaging.
\end{conj}

In \Cref{sec:pointw-conv}, we examine almost everywhere convergence,
and show in \Cref{bigthm:4} that a.e. convergence of the
\emph{difference} between the random averages and their expectations
follows if we add a mild speed condition to the assumption in \cref{eq:14}.

\subsection{Semirandom averages}
\label{sec:semirandom-averages}

Our third main topic is the convergence of double ergodic averages
with different rates. Combinatorial applications and norm-convergence
are also of interest, but the novelty is the pointwise convergence
result, since it's the first result for \emph{non-linear}
sequences. For linear sequences, it was \cite{MR1037434} who proved
that the averages $\frac1N \sum_{n<N} T^n F_1T^{2n}F_2 $ converge a.e.

Our main result for ``semirandom'' averages is
\begin{bigthm}[Semirandom pointwise averaging theorem]
  \label{bigthm:3}
  Suppose the decreasing sequence $\mathbbm u_1\ge\mathbbm
  u_2\ge\dots$ of probabilities satisfy
  \begin{equation}
    \label{eq:19}
    \mathbbm u_n\cdot \frac{n^{1/2}}{\log^{3+\delta}n}\to
    \infty\text{ for some } \delta>0,
  \end{equation}
  and
  \begin{equation}
    \label{eq:20}
    \mathbbm u_n\to 0.
  \end{equation}

  Then the random sequence $R^\omega=\cbrace{r_1<r_2<r_3<\dots}$
  satisfies, in every dynamical system $(X,\mu, T)$, for bounded
  functions $F_i$, that
  \begin{equation}
    \label{eq:21}
    \lim_{N\to\infty} \frac1N\sum_{n<N}T^nF_1(x)T^{r_n}F_2(x)=
    \overline F_1 \overline F_2\text{ for a.e. } x\in X,
  \end{equation}
  where $\overline F_i$ is the $L^2$ projection of $F_i$ to the
  $T$-invariant functions.
\end{bigthm}
\begin{rem} It will be clear from the proof that if we weaken the
  assumption in \cref{eq:19} to
  \begin{equation}
    \label{eq:22}
    \mathbbm u_n\cdot n^{1/2}\to \infty,
  \end{equation}
  we can still conclude mean convergence, and we have recurrence
  (intersectivity). What we mean by intersectivity in this context is
  that if the set of positive integers $A$ has positive upper density,
  then for some $n$ and $a$, the three numbers $a,a+n,a+r_n$ belong to
  $A$.  In fact, the lower density of such $n$'s is at least
  $\paren*{\overline d(A)}^3$
  \begin{equation}
    \label{eq:23}
    \liminf_{N\to\infty}\frac1N\sum_{n<N}\overline d\paren*{A\cap
      (A-n)\cap (A-r_n)}\ge \paren*{\overline d(A)}^3.
  \end{equation}

\end{rem}
In \cite{Fra}, we proved \cref{eq:21} with the condition $\mathbbm
u_n=n^{-b}$ for some $b<1/14$, so with our theorem, we improve the
range of $b$ to $b<1/2$.

Note that the assumption in \cref{eq:20} is equivalent with the random
sequence having $0$ density.  This assumption is \emph{necessary} to
have the limit equal $ \overline F_1\overline F_2$. If we drop the
assumption that $\mathbbm u_n\to 0$, then we don't know if we have
convergence---such as in the most basic ``coin flipping'' case,
\begin{uproblem}
  Suppose $\mathbbm u_n=1/2$.

  Is it true that the random sequence
  $R^\omega=\cbrace{r_1<r_2<r_3<\dots}$ satisfies, in every dynamical
  system and for bounded functions $F_i$, that the averages
  \begin{equation}
    \label{eq:24}
    \frac1N\sum_{n<N}T^nF_1(x)T^{r_n}F_2(x)
  \end{equation}
  converge for a.e. $x\in X$?
\end{uproblem}
\begin{conj}
  \label{sec:semirandom-averages-1}
  Suppose the decreasing sequence $\mathbbm u_1\ge\mathbbm
  u_2\ge\dots$ of probabilities satisfy
  \begin{equation}
    \label{eq:25}
    \mathbbm u_n\cdot \frac{n}{\paren*{\log\log n}^{1+\delta}}\to
    \infty\text{ for some } \delta>0,
  \end{equation}
  and
  \begin{equation}
    \label{eq:26}
    \mathbbm u_n\to 0.
  \end{equation}

  Then the random sequence $R^\omega=\cbrace{r_1<r_2<r_3<\dots}$
  satisfies, in every dynamical system, for bounded functions $F_i$,
  that
  \begin{equation}
    \label{eq:27}
    \lim_{N\to\infty} \frac1N\sum_{n<N}T^nF_1(x)T^{r_n}F_2(x)=
    \overline F_1 \overline F_2\text{ for a.e. } x\in X,
  \end{equation}
  where $\overline F_i$ is the projection of the function $F_i$ to the
  $T$-invariant functions.
\end{conj}
If $F_1$ is $T$-invariant, then the conjecture is true.  See
\cite{MR937581}. Even in this case, a speed of a power of $\log\log n$
is necessary as shown in \cite{MR1721622}.

\section{General ideas behind the proofs}
\label{sec:gener-cont-proof}

Our proof of \Cref{bigthm:1} will be presented in an elementary way,
but here we explain how the idea of the proof came about because we
believe this is a good general context to search for further
improvements.

In case of $1$- and $2$-intersectivity, the main tasks to prove are
\begin{align}
  \label{eq:28}
  \lim_{N\to\infty} \sup_{A} \abs*{\frac1{\sum_{n<N}\mathbbm
      u_n}\sum_{n< N} (U_n-\mathbbm u_n)\cdot \overline
    d\paren[\Big]{A\cap (A-n)}}&= 0
  \quad\text{with probability $1$},\\
  \intertext{and} \lim_{N\to\infty} \sup_{A}
  \abs*{\frac1{\sum_{n<N}\mathbbm u_n}\sum_{n< N} (U_n-\mathbbm
    u_n)\cdot \overline d\paren[\Big]{A\cap (A-n)\cap (A-2n)}}&= 0
  \quad\text{with probability $1$},
\end{align}
respectively, where in $\sup_A$ we take the supremum over all subsets
of the positive integers.  For $1$- and $2$-averaging the main tasks
are
\begin{align}
  \label{eq:29}
  \lim_{N\to\infty}\sup_{(X,T,\mu),F}\int_X\abs*{\frac1{\sum_{n<N}\mathbbm
      u_n}\sum_{n<N}(U_n-\mathbbm u_n) T^nF}&=0
  \text{ with probability 1, }\\
  \intertext{and}
  \lim_{N\to\infty}\sup_{(X,T,\mu),F_i}\int_X\abs*{\frac1{\sum_{n<N}\mathbbm
      u_n}\sum_{n<N}(U_n-\mathbbm u_n) T^nF_1\cdot T^{2n}F_2}&=0
  \text{ with probability 1, }
\end{align}
respectively, where in $\sup_{(X,T,\mu),F_i}$ we take the supremum
over all dynamical systems $(X,T,\mu)$ and indicators $F_i$.  After a
duality consideration, all these cases are proved via estimates of the
form
\begin{equation}
  \label{eq:30}
  P\,\paren*{\, \sup_{\paren*{a_1,a_2,a_3,\dots,a_{N-1}}\in\mathcal{A}}\abs*{\frac1{\sum_{n<N}\mathbbm u_n}\sum_{n< N}
      (U_n-\mathbbm u_n)\cdot a_n}>\epsilon }<
  \exp{\paren*{-c\epsilon^2\sum_{n<N}\mathbbm u_n }},
\end{equation}
where in $\sup_{\paren*{a_1,a_2,a_3,\dots,a_{N-1}}\in\mathcal{A}}$ we
take the the supremum over a certain class $\mathcal{A}$ of vectors
$\paren*{a_1,a_2,a_3,\dots,a_{N-1}}$ of real numbers with
$|a_n|\le1$. For example, in case of $2$-averaging, we have
\begin{equation}
  \label{eq:31}
  a_n=  \int_XG\cdot T^nF_1\cdot T^{2n}F_2
\end{equation}
and the supremum is taken over all dynamical systems $(X,T,\mu)$,
indicators $F_i$ and $\pm1$-valued functions $G$. The inequality in
\cref{eq:30} for a \emph{fixed} vector
$\paren*{a_1,a_2,a_3,\dots,a_{N-1}}$ with $|a_n|\le1$ follows from
Bernstein's exponential inequality, \Cref{lem:1} below, so the
difficulty is to handle the supremum. Now, for estimating the supremum
of a class of random variables $Y_k$, $k\in \mathcal{K}$, the best
estimate we have is the union estimate,
\begin{equation}
  \label{eq:32}
  P\,\paren*{ \sup_{k\in \mathcal{K}}Y_k > \epsilon}\le
  \abs*{\mathcal{K}}\cdot \sup_{k\in \mathcal{K}} P\,\paren*{ Y_k > \epsilon}.
\end{equation}
In our case, $\mathcal{K}$ would be the class of vectors
$\mathcal{A}=\paren*{a_1,a_2,a_3,\dots,a_{N-1}}$. This requires the reduction of
the number of vectors $\paren*{a_1,a_2,a_3,\dots,a_{N-1}}$ in the
supremum: we need to be able to find a subclass $\mathcal{A}_0$ of
$\mathcal{A}$ which has the following properties
\begin{description}
\item[$\mathcal{A}_0$ approximates $\mathcal{A}$ within $\epsilon$]
  for any $\paren*{a_1,a_2,a_3,\dots,a_{N-1}}\in \mathcal{A}$ there is
  $\paren*{a_1',a_2',a_3',\dots,a_{N-1}'}\in \mathcal{A}_0$ so that
  $|a_n-a_n'|<\epsilon$, $n=1,2,3,\dots, N$.
\item[The cardinality of $\mathcal{A}_0$ is
  $\exp{\paren*{o\paren*{\sum_{n<N}\mathbbm u_n }}}$ ] More precisely,
  for every $\epsilon$, we need to have $\log
  {\paren*{|\mathcal{A}_0|}}=O\paren*{ \epsilon^2\sum_{n<N}\mathbbm
    u_n}$.
\end{description}
In other words, the existence of $\mathcal{A}_0$ with the above
properties means that the $\epsilon$-covering number of $\mathcal{A}$
in the $\ell^\infty$ metric needs to be $\exp{\paren[\bigg]{O\paren*{
      \epsilon^2\sum_{n<N}\mathbbm u_n}}}$.  In both the intersective
and averaging case, $\mathcal{A}$ is uncountably infinite, but via
appropriate (and more or less standard) transference arguments, we
will show that in both cases $\mathcal{A}$ has finite
$\epsilon$-covering numbers.  In case of $1$-intersectivity and
$1$-averaging, the set $\mathcal{A}$ contains vectors
$\paren*{a_1,a_2,a_3,\dots,a_{N-1}}$ given by single correlation, that
is, by convolution,
\begin{equation}
  \label{eq:33}
  a_n=\frac1Q\sum_{a<Q}g(a)f(a+n),
\end{equation}
where $Q<N$, and $f$ and $g$ are two-valued functions defined on the
interval $[0,2N]$.  In this case, Parseval's formula easily yields
that the $\epsilon$-covering number is $O(N)=
\exp{\paren[\Big]{O\paren*{\log N}}}$ which implies that
$\sum_{n<N}\mathbbm u_n $ only needs to satisfy
\begin{equation}
  \label{eq:34}
  \frac{\sum_{n<N}\mathbbm u_n}{\log N}\to\infty.
\end{equation}
But in case of $2$-intersectivity and $2$-averaging, the structure of
$a_n$ is more complicated: the set $\mathcal{A}$ contains vectors
$\paren*{a_1,a_2,a_3,\dots,a_{N-1}}$ given by order $2$-correlation
\begin{equation}
  \label{eq:35}
  a_n=\frac1Q\sum_{a<Q}g(a)f_1(a+n)f_2(a+2n).
\end{equation}
The novelty of our paper is that even in this case we manage to show
that the $\epsilon$-covering number is better than the trivial $
\exp{\paren[\Big]{O\paren*{N}}}$ estimate: we show that for order
$\ell$ correlation sequences, the $\epsilon$-covering number is
$\exp{\paren[\bigg]{O\paren*{N^{1-1/(\ell +1}}}}$, which yields the
condition
\begin{equation}
  \label{eq:36}
  \frac{\sum_{n<N}\mathbbm u_n}{N^{1-1/(\ell +1)}}\to\infty,
\end{equation}
appearing in both \Cref{bigthm:1,bigthm:2}. We achieve this by
restricting the functions $f_i$ and $g$ to thin, \emph{independent}
subsets each of a certain appropriate density.  The independence of
these support sets results in restricting the order $\ell$ correlation
sequences to the \emph{product} of these densities.  The fact that we
have $\ell+1$ factors in an order $\ell$ correlation sequence explains
the appearance of $\ell+1$ in the condition \cref{eq:36}. The origin
of the idea of thinning via independent sets is coming from
\cite{2010arXiv1011.4310C}.  We later noticed that instead of using
independent random variables, we can simply use independent residue
classes and the Chinese remainder theorem, and this less sophisticated
method yields more accurate results.

We do believe, though, that even for order $\ell >1$ correlation
sequences, the $\epsilon$-covering numbers are not greater than some
power of $N$.  Hence the $\epsilon$-covering number is always
$\exp{\paren[\Big]{O\paren*{\log N}}}$, which would give the condition
\cref{eq:34}.  This belief is behind
\Cref{sec:introduction-history-1,sec:conv-ergod-aver-1}.

We close this section by stating that finding optimal
$\epsilon$-covering numbers is behind many convergence theorems.  The
general techniques of maximal inequalities, convexity methods
(cf. \Cref{sec:convexity-methods}), Fourier transform, spectral
theorem, uniform boundedness and transference principles can be viewed
as tools to help us reduce the original class of functions for which
we need to check convergence.

\section{Proof of \Cref{bigthm:1}, the random intersectivity theorem}
\label{sec:proof-crefbigthm:1}

The proof is completely elementary, and, unlike earlier methods, the
$\ell=2$ case contains all the ingredients of the proof. Hence, we
give the detailed proof for $\ell=2$, and only indicate the minor
adjustments to be made for $\ell \ge3$.

\subsection{The $\ell = 2$ case}
\label{sec:ell-=-1-1}

As we indicated, we shall work with the weaker assumption in
\cref{eq:8} which in this case is
\begin{equation}
  \label{eq:37}
  \frac{\sum_{n<N}\mathbbm u_n}{N^{2/3}}\to \infty,
\end{equation}
and we shall prove a little more than required: we shall prove that the
set of differences $r\in R^\omega$ for which $\overline
d\paren[\big]{A\cap(A-r)\cap(A-2r)}>0$ is of positive \emph{lower} density in
$R^\omega$,
\begin{equation}
  \label{eq:38}
  \liminf_{N\to\infty} \frac1{\sum_{n< N} U_n}\sum_{n< N} U_n\cdot \overline d\paren[\Big]{A\cap (A-n)\cap (A-2n)}>0.
\end{equation}
By the strong law of large numbers, \Cref{thm:1},
\begin{equation}
  \label{eq:39}
  \frac1{\sum_{n<N}\mathbbm u_n}\sum_{n<N}U_n\to 1 \text{ with probability 1, }
\end{equation}
so \cref{eq:38} is equivalent with
\begin{equation}
  \label{eq:40}
  \liminf_{N\to\infty} \frac1{\sum_{n<N}\mathbbm u_n}\sum_{n< N} U_n\cdot \overline d\paren[\Big]{A\cap (A-n)\cap (A-2n)}>0.
\end{equation}
Let us consider the \emph{expectation} of the averages
$\frac1{\sum_{n<N}\mathbbm u_n}\sum_{n< N} U_n\cdot \overline
d\paren[\Big]{A\cap (A-n)\cap (A-2n)}$,
\begin{equation}
  \label{eq:41}
  \frac1{\sum_{n<N}\mathbbm u_n}\sum_{n< N} \mathbbm u_n\cdot \overline d\paren[\Big]{A\cap (A-n)\cap (A-2n)}.
\end{equation}
Using that the sequence $(\mathbbm u_n)$ is decreasing\footnote{This
  is the only place where we use that the $(\mathbbm u_n)$ is
  decreasing.}, summation by parts gives the estimate
\begin{equation}
  \label{eq:42}
  \liminf_{N\to\infty} \frac1{\sum_{n<N}\mathbbm u_n}\sum_{n< N} \mathbbm
  u_n\cdot \overline d\paren[\Big]{A\cap (A-n)\cap (A-2n)}\ge \liminf_N  \frac1{N}\sum_{n< N} \overline d\paren[\Big]{A\cap (A-n)\cap (A-2n)}.
\end{equation}
Now, Szemerédi's (actually, in this $\ell=2$ case, Roth's) result has
been strengthened by \cite{MR0498471} to the extent that for any $A$
(of positive upper density) the \emph{lower} density of the differences of
three term arithmetic progressions in $A$ is positive
\begin{equation}
  \label{eq:43}
  \liminf_{N\to\infty} \frac1{N}\sum_{n< N} \overline d\paren[\Big]{A\cap (A-n)\cap (A-2n)}>0. 
\end{equation}
In fact, we need a more precise version of this result: Examining
Furstenber's transfer principle in \cite{MR0498471}, it is clear that
for a fixed $A$ of positive upper density, there is a \emph{single}
sequence $\seq{M}{k}$ so that for every $n$ the limit
  \begin{equation}
    \label{eq:44}
    \rho\paren[\Big]{A\cap(A-n)\cap(A-2n)}=\lim_{k\to\infty}\frac{\abs[\Big]{\cbrace[\big]{a\mid a\in A\cap(A-n)\cap(A-2n), \,a<M_k}}}{M_k}
  \end{equation}
exists and satisfies
\begin{equation}
  \label{eq:45}
  \liminf_{N\to\infty} \frac1{N}\sum_{n< N} \rho\paren[\Big]{A\cap (A-n)\cap (A-2n)}>0.
\end{equation}
Those who seek a more elementary reason for this can use the method of
\cite{MR0106865} to establish \cref{eq:45} and hence \cref{eq:43}.

It is then enough to prove that the difference between the averages
along the random sequence and the expectation of the averages goes to
$0$
\begin{equation}
  \label{eq:46}
  \lim_{N\to\infty} \sup_{A} \abs*{\frac1{\sum_{n<N}\mathbbm u_n}\sum_{n<
      N} (U_n-\mathbbm u_n)\cdot \rho\paren[\Big]{A\cap (A-n)\cap (A-2n)}}= 0
  \quad\text{with probability $1$},
\end{equation}
where in $\sup_A$ we take the supremum over \emph{all} subsets of the
positive integers.  

By the assumption in \cref{eq:37}, we can choose a sequence
$\epsilon_1, \epsilon_2,\dots,\epsilon_N,\dots$ of positive numbers so that
\begin{equation}
  \label{eq:47}
  \epsilon_N\to 0\text{ and }  \frac{\epsilon_N^2\sum_{n<N}\mathbbm
    u_n}{N^{2/3}}>C,\text{ with some constant $C>0$ to
    be chosen later, for
    all }N>N(C).
\end{equation}
We will show that, for each $N>N(C)$, the $N$th average is greater
than $\epsilon_N$ with small probability
\begin{equation}
  \label{eq:48}
  P\,\paren*{ \sup_{A}\abs*{\frac1{\sum_{n<N}\mathbbm u_n}\sum_{n< N}
      (U_n-\mathbbm u_n)\cdot \rho\paren[\Big]{A\cap (A-n)\cap (A-2n)}}>\epsilon_N }<
  \exp{\paren*{-c\epsilon_N^2\sum_{n<N}\mathbbm u_n }}.
\end{equation}
This implies \cref{eq:46} since, by the choice of $\epsilon_N$ in
\cref{eq:47}, we have $\exp{\paren*{-c \epsilon_N^2\sum_{n<N}\mathbbm
    u_n }}<\exp{\paren*{-Cc N^{2/3}}}$ which is summable in $N$, hence
we can apply the Borel-Cantelli lemma.

Now, the difficulty in proving \cref{eq:48} is the requirement that we
need to have the estimate of the average \emph{simultaneously} for all
$A$.  Indeed, the estimate for a \emph{fixed} set $A$ instead of
$\sup_{A}$ follows readily from Bernstein's exponential inequality,
\cite[page 52]{MR0388499},
\begin{lem}[Bernstein's exponential inequality]
  \label{lem:1}
  Let $X_1,X_2,\dots,X_N$ be independent, mean zero random variables with
  $\abs*{X_n}\le K$.

  Then we have
  \begin{equation}
    \label{eq:49}
    P\,\paren*{\abs*{\sum_{n< N} X_n}\ge t}\le
    2\max\cbrace[\Bigg]{\exp\paren*{-\frac{t^2/4}{\sum_{n< N}\mathbbm EX_n^2}},\,\,\,\exp\paren[\Big]{-t/(2K)}} \quad
    \text{ for all }\quad  t>0.
  \end{equation}
\end{lem}
\begin{cor}
  \label{cor:1}
  Let $\cbrace*{a_1,a_2,a_3,\dots,a_n,\dots}$ be a sequence of real or
  complex numbers with $|a_n|\le 1$.

  Then for any $\epsilon$, $0\le\epsilon<1$, we have
  \begin{equation}
    \label{eq:50}
    P\,\paren*{\, \abs*{\frac1{\sum_{n<N}\mathbbm u_n}\sum_{n< N}
        (U_n-\mathbbm u_n)\cdot a_n}>\epsilon }<
    \exp{\paren*{-\frac14\epsilon^2\sum_{n<N}\mathbbm u_n }}.
  \end{equation}
\end{cor}
\begin{proof}[Proof of \Cref{cor:1}]
  Set $X_n=(U_n-\mathbbm u_n)\cdot a_n$.  Then $\abs*{X_n}\le 1$ and
  \begin{align}
    \label{eq:51}
    \mathbb EX_n^2&\le \mathbb E(U_n-\mathbbm u_n)^2\\
    &\le \mathbbm u_n.
  \end{align}
  Now use \Cref{lem:1} with this $X_n$ and $t=\epsilon\cdot
  \sum_{n<N}\mathbbm u_n$.
\end{proof}
Using this corollary with $a_n=\rho\paren[\Big]{A\cap (A-n)\cap (A-2n)}$
and $\epsilon=\epsilon_N$, we readily get
\begin{equation}
  \label{eq:52}
  P\,\paren*{ \abs*{\frac1{\sum_{n<N}\mathbbm u_n}\sum_{n< N}
      (U_n-\mathbbm u_n)\cdot \rho\paren[\Big]{A\cap (A-n)\cap (A-2n)}}>\epsilon_N }<
  \exp{\paren*{-\frac14\epsilon_N^2\sum_{n<N}\mathbbm u_n }}.
\end{equation}
How can we get almost the same estimate when, instead a fixed $A$,
we take $\sup_{A}$? The best general estimate we have for the
supremum of a class of random variables $Y_k$, $k\in \mathcal{K}$, is
what is known as the union estimate
\begin{equation}
  \label{eq:53}
  P\,\paren*{ \sup_{k\in \mathcal{K}}Y_k > \epsilon}\le
  \abs*{\mathcal{K}}\cdot \sup_{k\in \mathcal{K}} P\,\paren*{ Y_k > \epsilon}.
\end{equation}
In view of the above, since in our case we would take $\mathcal{K}$ to
be a set of $A$'s, our task is to reduce the number of sets $A$
we need to consider in $\sup_{A}$.  Our first step is to show that
we can take \emph{finitely many} $A$. In calculating $\rho\paren[\Big]{A\cap (A-n)\cap (A-2n)}$, we are taking the limit of $\frac1{M_k}
\abs[\Big]{\cbrace[\big]{a\mid a\in A\cap (A-n)\cap (A-2n), \, a<M_k}}$ as $k\to\infty$. Our
formulas become simpler if we use the indicator $\mathbbm
1_A$ of $A$.  We can then write
\begin{equation}
  \label{eq:54}
  \frac1{M_k} \abs[\Big]{\cbrace[\big]{a\mid a\in A\cap (A-n)\cap(A-2n), \, a<M_k}}=\frac1{M_k}
  \sum_{a<M_k} \mathbbm 1_A(a)\mathbbm 1_A(a+n)\mathbbm 1_A(a+2n).
\end{equation}
Let us divide up the interval $[1,M_k)$ of summation into intervals of
length $Q$, where we'll see that the best choice for the positive
integer $Q$ will be a constant multiple of $N^{2/3}$,
\begin{equation}
  \label{eq:55}
  \begin{multlined}
    \frac1{M_k} \sum_{a<M_k} \mathbbm 1_A(a)\mathbbm 1_A(a+n)\mathbbm
    1_A(a+2n) \\
    = \frac1{M_k/Q}\sum_{i<M_k/Q}\frac1{Q}\sum_{iQ\le
      a<(i+1)Q} \mathbbm 1_A(a)\mathbbm 1_A(a+n)\mathbbm
    1_A(a+2n)+O\paren[\Big]{Q/M_k}.
  \end{multlined}
\end{equation}
It follows, since $M_k\to\infty$ hence the error term
$O\paren[\Big]{Q/M_k}$ goes to $0$, that
\begin{equation}
  \label{eq:56}
  \begin{multlined}
     \abs*{\frac1{\sum_{n<N}\mathbbm u_n}\sum_{n< N}
    (U_n-\mathbbm u_n)\cdot \rho\paren[\Big]{A\cap (A+n)\cap (A+2n)}}\\
  \le \sup_j \abs*{\frac1{\sum_{n<N}\mathbbm u_n} \sum_{n<
      N}(U_n-\mathbbm u_n)\cdot \frac1{Q}\sum_{jQ\le a<(j+1)Q}
    \mathbbm 1_A(a)\mathbbm 1_A(a+n)\mathbbm 1_A(a+2n)}.
  \end{multlined}
\end{equation}
Now, instead of taking the supremum $\sup_j$ over varying intervals of
length $Q$, we can take the \emph{fixed} interval $[0,Q)$, and take
the supremum over indicators over the intervals $[0,Q)$, $[0,Q+N)$ and
$[0,Q+2N)$. For this, simply change variable $b=a-jQ$, and define
\begin{equation}
\begin{aligned}
  f_0(b)&=\mathbbm 1_A(jQ+b)\mathbbm 1_{[0,Q)}(b), \\
  f_1(b)&=\mathbbm 1_A(jQ+b)\mathbbm
  1_{[0,Q+N)}(b),\\
  f_2(b)&=\mathbbm 1_A(jQ+b)\mathbbm 1_{[0,Q+2N)}(b).  
\end{aligned}\label{eq:57}
\end{equation}

It follows that
\begin{equation}
\begin{multlined}
  \sup_j \abs*{\frac1{\sum_{n<N}\mathbbm u_n} \sum_{n< N}(U_n-\mathbbm
    u_n)\cdot \frac1Q\sum_{jQ\le
      a<(j+1)Q} \mathbbm 1_A(a)\mathbbm 1_A(a+n)\mathbbm 1_A(a+2n)}\\
  \le \sup_{f_i\in \mathcal{F}_i} 
  \abs*{\frac1{\sum_{n<N}\mathbbm u_n} \sum_{n< N}(U_n-\mathbbm
    u_n)\cdot \frac1Q\sum_{b<Q} f_0(b)f_1(b+n)f_2(b+2n)},
\end{multlined}
\end{equation}

where $\mathcal{F}_i=\mathcal{F}_i(N)$ is the family of indicators
supported on the interval $[0,Q+iN)$ for $i=0,1,2$,
\begin{equation}
\label{eq:58}
  \mathcal{F}_i=\cbrace[\Big]{f_i\mid f_i:[0,Q+iN)\to
    \cbrace*{0,1}}, i=0,1,2.
\end{equation}
We've reduced the inequality in \cref{eq:48} to
\begin{equation}
  \label{eq:59}
  P\,\paren*{ \sup_{f_i\in \mathcal{F}_i}
    \abs*{\frac1{\sum_{n<N}\mathbbm u_n}
      \sum_{n< N}(U_n-\mathbbm u_n)\cdot \frac1{Q}\sum_{a<Q} f_0(a)f_1(a+n)f_2(a+2n)}>\epsilon_N }<
  \exp{\paren*{-c\epsilon_N^2\sum_{n<N}\mathbbm u_n }}.
\end{equation}
The important thing is that in $ \sup_{f_i\in \mathcal{F}_i}$,
$i=0,1,2$, we take the supremum over \emph{finite} sets.  The
cardinality of $\mathcal{F}_0$ is $2^{Q}$ and at this point it is
still at our disposal, so it is of no concern right now except we note
that $Q<N$.  But we do have a problem with the cardinalities of
$\mathcal{F}_i$, $i=1,2$: their cardinalities is at least
$2^{N}=\exp{\paren*{cN}}$. Applying \Cref{cor:1} with
$a_n=\frac1Q\sum_{a<Q} f_0(a)f_1(a+n)f_2(a+2n)$, and using the union
estimate \cref{eq:53}, we only get
\begin{equation}
  \label{eq:60}
  P\,\paren*{ \sup_{f_i\in \mathcal{F}_i}
    \abs*{\frac1{\sum_{n<N}\mathbbm u_n}
      \sum_{n< N}(U_n-\mathbbm u_n)\cdot \frac1Q\sum_{a<Q}
      f_0(a)f_1(a+n)f_2(a+2n)}>\epsilon_N }\le \exp{\paren*{cN}}\cdot \exp{\paren*{-\frac14\epsilon_N^2\sum_{n<N}\mathbbm u_n }},
\end{equation}
which is not good enough since the best estimate we have is
$\epsilon_N^2\sum_{n<N}\mathbbm u_n>C N^{2/3}$.  To reduce the
cardinality of the $f_1,f_2$ further, we restrict their support, the
$[0,Q+N)$ and $[0,Q+2N)$ intervals, to ``independent'' arithmetic
progressions.  We choose $2$ moduli $q_{1},q_2$,
and we split the support of $f_i$, the interval $[0,Q+ iN)$,
according to the residue classes $\pmod{q_i}$.  This results in
splitting the supports of the products $f_1(a+n)f_{2}(a+2n)$ into $q_{1}q_{2}$ residue classes: Let
$\chi_i$ be the indicator of the multiples of $q_i$
\begin{equation}
  \label{eq:61}
  \chi_i(x)=
  \begin{dcases*}
    1& if $q_i\mid x$,\\
    0& otherwise.
  \end{dcases*}
\end{equation}
Since $\chi_i(x-r)$ is the indicator of the $r\pmod {q_i}$ residue
class, we can write
\begin{equation}
  \label{eq:62}
  \begin{multlined}
    \frac1{Q}\sum_{a<Q}
    f_0(a)f_1(a+n)f_{2}(a+2 n)\\
    =\frac{1}{Q}
    \sum_{a<Q}f_0(a)\sum_{r_{1}<q_{1}}\chi_{1}(a+n-r_1)f_{1}(a+n)
    \sum_{r_{2}<q_2} \chi_2(a+2n-r_2)f_2(a+2n)
    \\
    =\frac{1}{Q} \sum_{r_{1}<q_{1}}
    \sum_{r_{2}<q_2}\sum_{a<Q}f_0(a)\chi_{1}(a+n-r_1)f_{1}(a+n)
    \chi_2(a+2 n-r_2)f_2(a+2 n).
  \end{multlined}
\end{equation}
We want the $q_i$ to be as big as possible, but we don't want the
number of terms $q_{1}q_{2}$ in the sum
$\sum_{r_{1}<q_{1}}\sum_{r_{2}<q_{2}}$ to
be greater than $Q$, so we require
\begin{equation}
  \label{eq:63}
  q_{1}q_{2}\le Q.
\end{equation}
We also need to make sure
$\sum_{a<Q}f_0(a)\chi_{1}(a+n-r_1)f_{1}(a+n)\chi_2(a+2
n-r_2)f_2(a+2n)$ is bounded independently of $n$ for fixed
$r_i$.  We can achieve this if the equation
\begin{equation}
  \label{eq:64}
  \chi_{1}(a+n-r_1)\chi_{2}(a+2n-r_2)=1
\end{equation}
has, say, at most $1$ solution in $a\in [0,Q)$.  We can rewrite
\cref{eq:64} as simultaneous congruences
\begin{equation}\label{eq:65}
  \begin{aligned}   
    a&\equiv r_1-n &&\pmod{q_{1}},\\
    a&\equiv r_2-2n &&\pmod{q_{2}}.
  \end{aligned}
\end{equation}
Now, if the $q_i$ are coprimes, then the above system of
congruences has a single solution $\pmod{q_{1}q_{2}}$ by
the Chinese remainder theorem.  It follows that \cref{eq:64} has at
most $1$ solution in $a\in [0,Q)$ if
\begin{equation}
  \label{eq:66}
  Q\le q_{1}q_{2}.
\end{equation}
Comparing this with \cref{eq:63}, we commit to
\begin{equation}
  \label{eq:67}
  Q= q_{1}q_{2}.
\end{equation}
We then have
\begin{align}
  \label{eq:68}
  &\abs*{\frac1{\sum_{n<N}\mathbbm u_n}
      \sum_{n< N}(U_n-\mathbbm u_n)\cdot \frac1{Q}\sum_{a<Q}
      f_0(a)f_1(a+n)f_2(a+2n)}\\
    \intertext{by \cref{eq:62} and since $Q=q_1q_2$}
    &= \abs*{\frac1{Q}\sum_{r_{1}<q_{1}}
    \sum_{r_{2}<q_2}\frac1{\sum_{n<N}\mathbbm u_n}
      \sum_{n< N}(U_n-\mathbbm u_n)\cdot \sum_{a<Q}f_0(a)\chi_{1}(a+n-r_1)f_{1}(a+n)
      \chi_2(a+2 n-r_2)f_2(a+2 n)}\\
    &\le \sup_{r_i<q_i}\abs*{\frac1{\sum_{n<N}\mathbbm u_n}
      \sum_{n< N}(U_n-\mathbbm u_n)\cdot \sum_{a<Q}f_0(a)\chi_{1}(a+n-r_1)f_{1}(a+n)
      \chi_2(a+2 n-r_2)f_2(a+2 n)}.
\end{align}
It follows, using the union estimate \cref{eq:53},
\begin{equation}
   \begin{multlined}
   P\paren*{\sup_{f_i\in \mathcal{F}_i}\abs*{\frac1{\sum_{n<N}\mathbbm u_n}
      \sum_{n< N}(U_n-\mathbbm u_n)\cdot
      \sum_{a<Q}f_0(a)f_{1}(a+n)f_2(a+2 n)}>\epsilon_N}\\
  \le q_1q_2\sup_{r_i<q_i}P\paren*{\sup_{f_i\in \mathcal{F}_i}\abs*{\frac1{\sum_{n<N}\mathbbm u_n}
      \sum_{n< N}(U_n-\mathbbm u_n)\cdot \sum_{a<Q}f_0(a)\chi_{1}(a+n-r_1)f_{1}(a+n)
      \chi_2(a+2 n-r_2)f_2(a+2 n)}>\epsilon_N}.
  \end{multlined}
\end{equation}
The contribution of $q_1q_2=Q<N= \exp{\paren*{ \log N}}$ is insignificant to
the order of $\exp{\paren*{-\epsilon_N^2\sum_{n<N}\mathbbm
    u_n}}=\exp{\paren*{-O(N^{2/3})}}$ which we aim to achieve in
\cref{eq:59}, so we ignore it in further considerations. For given
$r_i$, $i=1,2$, the function $g_i$ defined by $g_i(x)=\chi(x-r_i)f_i(x)$ is supported on
the part of the residue class $r_i\pmod{ q_i}$ that falls into the interval
$[0,Q+iN)$.  Denote by $\mathcal{G}_i$ those functions from $\mathcal{F}_i$
which are supported on the residue class $r_i\pmod{ q_i}$
\begin{equation}
  \mathcal{G}_i=\cbrace[\Big]{g_i\mid g_i:[0,Q+iN)\to \cbrace*{0,1},
    g_i(x)\ne0\implies x\equiv r_i\pmod {q_i}}.
\end{equation}
Since the cardinality of the residue class $r_i\pmod {q_i}$ that falls into
the interval $[0,Q+iN)$ is not more than $(Q+iN)/q_i<3N/q_i$, we have
\begin{equation}
  \abs*{\mathcal{G}_i}\le 2^{2N/q_i}, i=1,2.\label{eq:69}
\end{equation}
Our remaining task is to estimate, using the convenient notation $\mathcal{G}_0=\mathcal{F}_0$,
\begin{equation}
  \label{eq:70}
  P\paren*{\sup_{g_i\in \mathcal{G}_i}\abs*{\frac1{\sum_{n<N}\mathbbm u_n}
      \sum_{n< N}(U_n-\mathbbm u_n)\cdot \sum_{a<Q}g_0(a)g_{1}(a+n)
      g_2(a+2 n)}>\epsilon_N}.
\end{equation}
In the above, $\abs*{\sum_{a<Q}g_0(a)g_{1}(a+n) g_2(a+2 n)}$ is
bounded by $1$ for each $n$.  This is because $q_1q_2=Q$, $q_1$ and $q_2$
are coprimes, hence, by the Chinese remainder theorem, there is one
$a<Q$ for which $a+n\equiv r_1\pmod {q_1}$ and $a+2n\equiv r_2\pmod
{q_2}$. Clearly
\begin{equation}
  \label{eq:71}
  \abs*{\mathcal{G}_0}= 2^{Q},
\end{equation}
hence, using \cref{eq:69},
\begin{equation}
  \label{eq:72}
  \abs*{\mathcal{G}_0\times \mathcal{G}_1 \times \mathcal{G}_2}\le 2^{Q+2N/q_1+2N/q_2}.
\end{equation}
We can now estimate, using the union estimate \cref{eq:53} and
\Cref{cor:1} with $a_n=\sum_{a<Q}g_0(a)g_{1}(a+n) g_2(a+2 n)$, as
\begin{equation}
  \label{eq:73}
  \begin{multlined}
    P\paren*{\sup_{g_i\in
        \mathcal{G}_i}\abs*{\frac1{\sum_{n<N}\mathbbm u_n} \sum_{n<
          N}(U_n-\mathbbm u_n)\cdot \sum_{a<Q}g_0(a)g_{1}(a+n) g_2(a+2
        n)}>\epsilon_N}\\
    \le 2^{Q+2N/q_1+2N/q_2}\cdot \exp{\paren*{-\frac14\epsilon_N^2
        \sum_{n<N}\mathbbm u_n}}\\
    \le \exp{\paren*{(Q+2N/q_1+2N/q_2)-\frac14\epsilon_N^2 \sum_{n<N}\mathbbm
      u_n}},
  \end{multlined}
\end{equation}
where we used that $\log 2<1$. In order to have a nontrivial estimate,
we need to have something like
\begin{equation}
  \label{eq:74}
  Q+\frac{2N}{q_1}+ \frac{2N}{q_2}\le \frac1{5}\cdot \epsilon_N^2
  \sum_{n<N}\mathbbm u_n.
\end{equation}
Treating the left hand side as if $q_i$ and $Q$ were real variables,
since $Q=q_1q_2$, the left hand side has a minimum at $Q=cN^{2/3}$,
and $q_1=q_2=cN^{1/3}$.  We certainly can choose coprime integers
$q_i$ which nearly satisfy these requirements.  With these choices
\cref{eq:74} becomes
\begin{equation}
  \label{eq:75}
  cN^{2/3} \le \frac1{5}\cdot \epsilon_N^2
  \sum_{n<N}\mathbbm u_n.
\end{equation}
We certainly can achieve this by taking $C$ large enough in
\cref{eq:47}. With these choices of $C,Q,q_i$, the estimate in
\cref{eq:73} becomes
\begin{equation}
  \label{eq:76}
  P\paren*{\sup_{g_i\in
        \mathcal{G}_i}\abs*{\frac1{\sum_{n<N}\mathbbm u_n} \sum_{n<
          N}(U_n-\mathbbm u_n)\cdot \sum_{a<Q}g_0(a)g_{1}(a+n) g_2(a+2
        n)}>\epsilon_N}\le \exp{\paren*{-cN^{2/3}}},
\end{equation}
finishing our proof.

\subsection{The general $\ell$ case}
\label{sec:general-ell-case}

There are only three remarks we need to make here, and then the proof
goes through verbatim. 

The first remark is that in \cite{MR0498471} \cref{eq:34} is proved
for all $\ell$:
\begin{equation}
  \label{eq:77}
  \liminf_{N\to\infty} \frac1{N}\sum_{n< N} \overline d\paren[\Big]{A\cap
    (A-n)\cap \dots \cap  (A-(\ell-1)n)\cap (A-\ell n)}>0.
\end{equation}
For an elementary proof: N. Hegyvári pointed out to us that the method
of \cite{MR0106865} can be used for the same conclusion\footnote{In
  \cite{MR0106865}, only the $\ell=2$ case is treated, but the method
  works for any $\ell$. For a proof for any $\ell$, see Theorem $8'$
  in
  \url{http://terrytao.wordpress.com/2008/08/30/the-correspondence-principle-and-finitary-ergodic-theory/}.}.
The slightly strengthened version we really need, and which we
described in \cref{eq:45}, is also valid.

Second, this time we need to choose the $(\epsilon_N)$ so that
\begin{equation}
  \epsilon_N\to 0\text{ and }  \frac{\epsilon_N^2\sum_{n<N}\mathbbm
    u_n}{N^{1-1/(\ell+1)}}>C,\text{ with some constant $C>0$ to
    be chosen later, for
    all }N>N(C).
\end{equation}

For the third remark, let us look at the essential inequality we
really need to prove: with a $Q<N$ to be chosen later, we need to
prove the following analog of \cref{eq:59}
\begin{equation}
  \begin{multlined}
    P\,\paren*{ \sup_{f_i\in \mathcal{F}_i}
    \abs*{\frac1{\sum_{n<N}\mathbbm u_n} \sum_{n< N}(U_n-\mathbbm
      u_n)\cdot \frac1{Q}\sum_{a<Q}
      f_0(a)f_1(a+n)\dots f_{\ell-1}(a+(\ell-1) n)f_{\ell}(a+\ell n)}>\epsilon_N }\\
  < \exp{\paren*{-c\epsilon_N^2\sum_{n<N}\mathbbm u_n }},
  \end{multlined}
\end{equation}
where
\begin{equation}
   \mathcal{F}_i=\cbrace[\Big]{f_i\mid f_i:[0,Q+iN)\to
    \cbrace*{0,1}}, i=0,1,2,\dots,\ell.
\end{equation}
This time, we need to slim down the supports of $\ell$ functions
$f_{1},f_{2},\dots,f_{\ell}$. Correspondingly, we choose $\ell$
pairwise coprime moduli $q_{1},q_2,\dots,q_{\ell}$ to split the
support of $f_i$, the interval $[0,Q+i N)$, according to the residue
classes $\pmod{q_i}$.  Proceeding as in the case $\ell=2$, we get the
requirements
\begin{equation}
  \label{eq:78}
  q_{1}q_{2}\dots q_\ell= Q, 
\end{equation}
and
\begin{equation}
  \label{eq:79}
  \frac{\ell N}{q_{1}}+ \frac{\ell N}{q_{2}}+\dots+\frac{\ell N}{q_{\ell}}+Q\le
  \frac15\epsilon_N^2 \sum_{n<N}\mathbbm u_n.
\end{equation}
If the $q_i$ and $Q$ were real variables, the left hand side would be
minimized when
\begin{equation}
  Q=cN^{\ell/(\ell +1)},
\end{equation}
and the $q_i$ were equal to each other.  By \cref{eq:78} this means
\begin{equation}\label{eq:80}
  q_i=Q^{1/\ell}=cN^{1/(\ell+1)}.
\end{equation}
With these values, \cref{eq:79} becomes, with some $c$ depending on
$\ell$ only,
\begin{equation}
  \label{eq:81}
  cN^{1-1/(\ell +1)}<\epsilon_N^2 \sum_{n<N}\mathbbm u_n, 
\end{equation}
which suggested the assumption in \cref{eq:8}.  The only question
remains if we can choose pairwise coprime $q_i$'s near the optimal
value $N^{1/(\ell+1)}$ of \cref{eq:80}?  To do this, choose each $q_i$
to be a prime number with $N^{1/(\ell +1)}<q_i<2N^{1/(\ell +1)}$.
This is possible for large enough $N$ by the prime number theorem.
The much simpler Chebyshev estimate $\frac12 t/\log t <\pi(t)< 2t/\log
t$ is also enough, but then the primes would satisfy $N^{1/(\ell
  +1)}<q_i<5N^{1/(\ell +1)}$.

\section{Proof of \Cref{bigthm:2}, the random averaging theorem}
\label{sec:proof-}

Similarly to the proof of \Cref{bigthm:1}, we just explain the idea of
the proof for $\ell=2$.  Since the proof is very similar to the proof
of \Cref{bigthm:1}, we just point out the differences between the
arguments.

\subsection{The $\ell=2$   case}
\label{sec:ell=1-case}

We want to show that with probability $1$ in any dynamical system 
\begin{equation}
  \label{eq:82}
  \frac1{\sum_{n<N}U_n}\sum_{n<N}U_n\cdot T^nF_1\cdot T^{2n}F_2 \text{ converge in $L^1$ norm.}
\end{equation}

Approximating $F_i$ in $L^1$ norm by linear combinations of indicator
functions, we see it's enough to show convergence for indicators
$F_i$ of measurable subsets of $X$. By the strong law of large numbers, \Cref{thm:2}, we have
\begin{equation}
  \label{eq:83}
  \frac{\sum_{n<N}U_n}{\sum_{n<N}\mathbbm u_n}\to 1 \text{ with probability 1, }
\end{equation}
hence \cref{eq:82} is equivalent to
\begin{equation}
  \label{eq:84}
  \frac1{\sum_{n<N}\mathbbm u_n}\sum_{n<N}U_n\cdot T^nF_1\cdot T^{2n}F_2 \text{ converge in $L^1$ norm.}
\end{equation}
The expectation of the averages $\frac1{\sum_{n<N}\mathbbm
  u_n}\sum_{n<N}U_n\cdot T^nF_1\cdot T^{2n}F_2 $ is
\begin{equation}
  \label{eq:85}
  \frac1{\sum_{n<N}\mathbbm u_n}\sum_{n<N}\mathbbm u_n\cdot T^nF_1\cdot T^{2n}F_2.
\end{equation}
Summation by parts shows that the averages in \cref{eq:85} converge to
the same limit as the double ergodic averages
$\frac1{N}\sum_{n<N}T^nF_1\cdot T^{2n}F_2$.  That these double averages
converge in norm was proved by \cite{MR0498471}, so it is enough to
prove that the differences between the random averages and their
expectations go to $0$ in $L^1$ norm
\begin{equation}
  \label{eq:86}
  \lim_{N\to\infty}\sup_{(X,T,\mu),F_i}\int_X\abs*{\frac1{\sum_{n<N}\mathbbm u_n}\sum_{n<N}(U_n-\mathbbm u_n) \cdot T^nF_1\cdot T^{2n}F_2}=0
  \text{ with probability 1, }
\end{equation}
where in $\sup_{(X,T,\mu),F_i}$ we take the supremum over all dynamical
systems $(X,T,\mu)$ and indicators $F_i$, $1=1,2$.

We first reduce all to a single dynamical system using an appropriate
version of the transference principle of A. Calderón.  Since $T$ is measure
preserving, we have, for every integer $a$
\begin{equation}
  \label{eq:87}
  \int_X\abs*{\frac1{\sum_{n<N}\mathbbm u_n}\sum_{n<N}(U_n-\mathbbm
    u_n)\cdot T^nF_1\cdot T^{2n}F_2 } = \int_X\abs*{\frac1{\sum_{n<N}\mathbbm u_n}\sum_{n<N}(U_n-\mathbbm u_n) \cdot T^{a+n}F_1 T^{a+2n}F_2}.
\end{equation}
Average this in $a$ over the interval $[0,Q)$, where $Q$ is the same
as in the proof of \Cref{bigthm:1} in the $\ell=2$ case, so its best
choice will be $Q\sim N^{2/3}$. We get
\begin{equation}
  \label{eq:88}
  \int_X\abs*{\frac1{\sum_{n<N}\mathbbm u_n}\sum_{n<N}(U_n-\mathbbm
    u_n) \cdot T^nF_1\cdot T^{2n}F_2 } = \int_X\frac1Q\sum_{a<Q}\abs*{\frac1{\sum_{n<N}\mathbbm u_n}\sum_{n<N}(U_n-\mathbbm u_n) \cdot T^{a+n}F_1 T^{a+2n}F_2}.
\end{equation}
It follows that
\begin{equation}
  \label{eq:89}
  \sup_{F_i}\int_X\abs*{\frac1{\sum_{n<N}\mathbbm u_n}\sum_{n<N}(U_n-\mathbbm
    u_n) \cdot T^nF_1\cdot T^{2n}F_2 } \le \int_X\sup_{F_i}\frac1Q\sum_{a<Q}\abs*{\frac1{\sum_{n<N}\mathbbm u_n}\sum_{n<N}(U_n-\mathbbm u_n) \cdot T^{a+n}F_1 T^{a+2n}F_2}.
\end{equation}
Define the $\mathbbm N\to\cbrace{0,1}$ indicators $f_1,f_2$ by
\begin{equation}
  \label{eq:90}
  \begin{aligned}
    f_1(a)=f_1(x,a)&=T^{a}F_1(x)\mathbbm 1_{[0,Q+N)}(a),\\
    f_2(a)=f_2(x,a)&=T^{a}F_2(x)\mathbbm 1_{[0,Q+2N)}(a).
  \end{aligned}
\end{equation}
Then for the values of $a$ and $n$ considered above we have
$T^{a+n}F_1(x)=f_1(a+n)$ and $T^{a+2n}F_1(x)=f_2(a+2n)$.  Denote
by $\mathcal F_i$ the family of indicators supported on
the interval $[0,Q+iN)$
\begin{equation}
  \label{eq:91}
  \mathcal{F}_i=\cbrace[\Big]{f_i\mid f_i:[0,Q+iN)\to
    \cbrace*{0,1}}.
\end{equation}
We then have the pointwise estimate
\begin{equation}
  \label{eq:92}
  \begin{multlined}
    \sup_{F_i}\frac1Q\sum_{a<Q}\abs*{\frac1{\sum_{n<N}\mathbbm
        u_n}\sum_{n<N}(U_n-\mathbbm u_n) \cdot T^{a+n}F_1(x)
      T^{a+2n}F_2(x)} \\
    \le \sup_{f_i\in
      \mathcal{F}_i}\frac1Q\sum_{a<Q}\abs*{\frac1{\sum_{n<N}\mathbbm
        u_n}\sum_{n<N}(U_n-\mathbbm u_n) \cdot f_1(a+n)f_2(a+2n)}.
  \end{multlined}
\end{equation}
Integrating this over the probability space $X$, and using
\cref{eq:89}, we get
\begin{equation}
  \label{eq:93}
  \begin{multlined}
    \sup_{(X,T,\mu)}\sup_{F_i}\int_X\abs*{\frac1{\sum_{n<N}\mathbbm
        u_n}\sum_{n<N}(U_n-\mathbbm u_n) \cdot T^{n}F_1
      T^{2n}F_2} \\
    \le  \int_X\sup_{f_i\in
      \mathcal{F}_i}\frac1Q\sum_{a<Q}\abs*{\frac1{\sum_{n<N}\mathbbm
        u_n}\sum_{n<N}(U_n-\mathbbm u_n)\cdot f_1(a+n)f_2(a+2n) }\\
    =\sup_{f_i\in
      \mathcal{F}_i}\frac1Q\sum_{a<Q}\abs*{\frac1{\sum_{n<N}\mathbbm
        u_n}\sum_{n<N}(U_n-\mathbbm u_n)\cdot f_1(a+n)f_2(a+2n) }
  \end{multlined}
\end{equation}
The important thing to notice is that the final estimate of
\cref{eq:93} doesn't depend on the dynamical system $(X,T,\mu)$
anymore and, furthermore, in $\sup_{f_i\in \mathcal{F}_i}$ we take the
supremum over finitely many functions $f_i$, $i=1,2$. Our remaining
task is to show that
\begin{equation}
  \label{eq:94}
  \lim_{N\to\infty} \sup_{f_i\in \mathcal{F}_i}\frac1Q\sum_{a<Q}\abs*{\frac1{\sum_{n<N}\mathbbm
      u_n}\sum_{n<N}(U_n-\mathbbm u_n) \cdot f_1(a+n)f_2(a+2n)}=0  \text{ with probability 1. }
\end{equation}
Let us write the above average in dual form: Let $\mathcal G=\mathcal
G(N)$ be the family of $\pm 1$-valued functions supported on the
interval $[0,Q)$
\begin{equation}
  \label{eq:95}
  \mathcal{G}=\cbrace[\Big]{g\mid g:[0,Q)\to \cbrace*{-1,1}}.
\end{equation}
We then have
\begin{equation}
  \label{eq:96}
  \begin{multlined}
    \sup_{f_i\in
      \mathcal{F}_i}\frac1Q\sum_{a<Q}\abs*{\frac1{\sum_{n<N}\mathbbm
        u_n}\sum_{n<N}(U_n-\mathbbm u_n) \cdot f_1(a+n)f_2(a+2n)}\\
    =\sup_{f_i\in \mathcal{F}_i,g\in
      \mathcal{G}}\frac1Q\sum_{a<Q}\frac1{\sum_{n<N}\mathbbm
      u_n}\sum_{n<N}(U_n-\mathbbm u_n) \cdot g(a) f_1(a+n)f_2(a+2n)\\
    =\sup_{f_i\in \mathcal{F}_i,g\in
      \mathcal{G}}\frac1{\sum_{n<N}\mathbbm
      u_n}\sum_{n<N}(U_n-\mathbbm u_n) \frac1Q\sum_{a<Q}g(a) f_1(a+n)f_2(a+2n).
  \end{multlined}
\end{equation}
The only difference between the right hand side of \cref{eq:96} above
and
\begin{equation}
  \label{eq:97}
  \sup_{f_i\in
    \mathcal{F}_i}\frac1{\sum_{n<N}\mathbbm
    u_n}\sum_{n<N}(U_n-\mathbbm u_n) \frac1Q\sum_{a<Q}f_0(a) f_1(a+n)f_2(a+2n)
\end{equation}
of \cref{eq:59} is that $\mathcal{G}$ contains $\pm 1$-valued
functions while the $\mathcal{F}_0$ of \cref{eq:97} contains
indicators.  But this difference doesn't change anything in the proof,
since the only thing we use about the functions in $\mathcal{G}$ is
that they take on two values. So the method used to establish
\cref{eq:59} works equally well to prove, for all large enough $ N $,
\begin{equation}
  \label{eq:98}
  P\,\paren*{\sup_{f_i\in \mathcal{F_i},g\in \mathcal{G}}\frac1{\sum_{n<N}\mathbbm
      u_n}\sum_{n<N}(U_n-\mathbbm u_n) \frac1Q\sum_{a<Q} g(a) f_1(a+n)f_2(a+2n) >\epsilon_N }<
  \exp{\paren*{-c\epsilon_N^2\sum_{n<N}\mathbbm u_n }},
\end{equation}
where $\epsilon_N$ is chosen to satisfy \cref{eq:47}.  But this choice
of $\epsilon_N$ implies that the right hand side of \cref{eq:98} is
summable and hence, by the Borel-Cantelli lemma, we have \cref{eq:94}.

\subsection{The general $\ell$ case}
\label{sec:general-ell-case-1}

Our remarks are similar to the three remarks we made in
\Cref{sec:general-ell-case}.  The only difference is in the first
remark: While in the $\ell=2$ case, we invoked Furstenberg's double
convergence theorem, this time we need to refer to the ``multiple mean
ergodic theorem'' of \cite{MR2150389} (cf \cite{MR2257397} for another
proof).  The multiple mean ergodic theorem says that the set of
positive integers is $\ell$-averaging, that is the averages
\begin{equation}
  \label{eq:99}
  \frac1N\sum_{n<N}T^{n}F_{1}\cdot T^{2
    n}F_{2}\cdots T^{\ell n}F_{\ell}
\end{equation}
converge in $L^1$ norm for bounded $F_i$. Then the proof proceeds to
show that the averages along the random sequence converge to the same
limit as the multiple averages above.

\subsection{Pointwise convergence}
\label{sec:pointw-conv}

The main philosophy on pointwise convergence for \emph{any} averages
(not just those in our paper) is that if mean convergence takes place
with an appropriate speed, then we also have pointwise convergence. It
is enough to prescribe the speed for a \emph{subsequence} of the
averages
\begin{equation}
  \label{eq:100}
  \frac1{\sum_{n< N} U_n }\sum_{n< N} (U_n-\mathbbm u_n)T^{n}F_{1}(x)\cdot T^{2
    n}F_{2}(x)\cdots T^{\ell n}F_{\ell}(x).
\end{equation}

What is the sparsest subsequence of the averages the convergence of
which implies the convergence of all the averages? Actually, we are
going to consider a \emph{class} of subsequences: for $\sigma>1$, let
$\cbrace{N_1<N_2<N_3<\dots}$ be such that
\begin{equation}
  \label{eq:101}
  \sum_{n<N_i} \mathbbm u_n \approx \sigma^i.
\end{equation}
Roughly speaking, for each $\sigma>1$, the $N_i$ defined above is a
sequence along which the averages could still change significantly,
but as $\sigma$ gets closer and closer to $1$, this change is as small
as we want.  To understand why, consider the simplest random average
relevant to our problem
\begin{equation}
  \label{eq:102}
  \frac1{\sum_{n< N} \mathbbm u_n }\sum_{n< N} U_n.
\end{equation}
The $N_i$th sum $ \sum_{n< N_i} U_n$ has an expected mass of $
\sum_{n< N_i} \mathbbm u_n\approx \sigma^i$ and the $N_{i+1}$st sum $
\sum_{n< N_{i+1}} U_n$ has an expected mass of $ \sum_{n< N_{i+1}}
\mathbbm u_n\approx \sigma^{i+1}$. Hence the difference $ \sum_{N_i\le
  n< N_{i+1}} U_n$ has an expected mass of $ \sum_{N_i\le n< N_{i+1}}
\mathbbm u_n\approx \sigma^{i+1}-\sigma^i$. What is the proportion of
this increase in mass to $\sigma^i$, the expected mass of the $N_i$th
average? It is $(\sigma^{i+1}-\sigma^i)/\sigma^i=\sigma-1$. This
means, the proportion of this difference is a fixed positive number
(independent of $i$), so if changes happen in this difference, it may
affect averages that are between the $N_i$th and $N_{i+1}$st average.
On the other hand, the proportion $\sigma-1$ of this difference goes
to $0$ as $\sigma$ gets closer to $1$.  We then expect that if the
$N_i$th average and the $N_{i+1}$st averages are both close to $1$,
and if $N$ is between $N_i$ and $N_{i+1}$, then the $N$th average will
also be close to $1$, and the difference will not be much more than
$\sigma-1$.

Let us now carry out this argument for our averages precisely.  Our
random average in \cref{eq:100} is of the form
\begin{equation}
  \label{eq:103}
  \frac1{\sum_{n< N} U_n} \sum_{n<N} (U_n-\mathbbm u_n)G_n
\end{equation}
for some uniformly bounded sequence of functions $G_1,G_2,G_3,\dots$.
We can assume $\norm{G_n}_\infty\le 1$.  By the strong law of large
numbers, \Cref{thm:2}, with probability $1$, we can divide by the
expectation $\sum_{n<N} \mathbbm u_n$ instead of $\sum_{n< N} U_n$
\begin{equation}
  \label{eq:104}
  \frac1{\sum_{n< N} \mathbbm u_n} \sum_{n<N} (U_n-\mathbbm u_n)G_n.
\end{equation}
For $\sigma>1$, we choose the subsequence
$I_\sigma=\cbrace{N_1<N_2<N_3<\dots}$ so that $\sum_{n<N_i}\mathbbm
u_n\approx \sigma^i$,
\begin{equation}
  \label{eq:105}
  \sum_{n<N_i} \mathbbm u_n \ge \sigma^i \text{ but }
  \sum_{n<N} \mathbbm u_n < \sigma^i \text{ for }N<N_i.
\end{equation}
In words: $N_i$ is the smallest index $N$ for which $\sum_{n<N}
\mathbbm u_n\ge \sigma^i$.  That $I_\sigma$ is well defined and
infinite follows from the dissipative assumption $\sum_n \mathbbm
u_n=\infty$. Clearly
\begin{equation}
  \label{eq:106}
  \sum_{n<N_i}\mathbbm
  u_n= \sigma^i+O(1)\text{ for all }i=1,2,3,\dots.
\end{equation}
Note that we also have that $\sum_{n<N_i}U_n\approx \sigma^i$. More
precisely, we have, by the strong law of large numbers, \Cref{thm:2},
and by \cref{eq:106},
\begin{equation}
  \label{eq:107}
  \sum_{n<N_i}U_n= \sigma^i + o\paren*{\sigma^i}\text{ with
    probability }1.
\end{equation}
Let us see how convergence along the subsequences $ I_\sigma$'s,
$\sigma>1$, imply convergence along the full sequence
$N=1,2,3,\dots$. Assume that  for each fixed $\sigma>1$, there
is $\Omega_\sigma$ with $P(\Omega_\sigma)=1$ so that 
\begin{equation}
  \label{eq:108}
  \lim_{i\to\infty} \frac1{\sum_{n< N_i} \mathbbm u_n} \sum_{n<N_i}
  (U_n(\omega)-\mathbbm u_n)G_n(x) = 0\text{  for }\omega\in
  \Omega_\sigma\text{  and for a.e. }x\in X. 
\end{equation}
We want to show that then for each fixed $\sigma>1$, there
is $\Omega_\sigma'$ with $P(\Omega_\sigma')=1$ so that 
\begin{equation}
  \label{eq:109}
  \limsup_{N\to\infty}\abs*{\frac1{\sum_{n< N} \mathbbm u_n} \sum_{n<N}
    (U_n(\omega)-\mathbbm u_n)G_n(x)}\le 4(\sigma-1) \text{  for }\omega\in \Omega_\sigma'\text{  and for a.e. }x\in X.
\end{equation}
It's important to emphasize that the construction of the set
$\Omega_\sigma'$ is independent of $x\in X$, and that the only
property of the sequence $G_1(x), G_2(x),\dots,G_n(x),\dots$ we use is
that it's bounded by $1$. For this reason, we suppress the argument
$x$ in $G_n(x)$.

Now, if we choose a \emph{sequence} of $\sigma$'s which go to $1$,
then the intersection $\Omega'$ of the corresponding sets
$\Omega_\sigma'$ for which \cref{eq:109} holds will give convergence
to $0$,
\begin{equation}
  \label{eq:110}
  \lim_{N\to\infty} \frac1{\sum_{n< N} \mathbbm u_n} \sum_{n<N}
  (U_n(\omega)-\mathbbm u_n)G_n = 0\text{  for }\omega\in \Omega'.
\end{equation}
So fix $\sigma>1$.  By the assumption in \cref{eq:108}, \cref{eq:109}
follows if we show that there is $\Omega_\sigma'$ with
$P(\Omega_\sigma')=1$ so that
\begin{equation}
  \label{eq:111}
  \limsup_{i\to\infty}\max_{N_i\le N<N_{i+1}}\abs*{\frac1{\sum_{n< N}
      \mathbbm u_n} \sum_{n<N} (U_n(\omega)-\mathbbm u_n)G_n-\frac1{\sum_{n<
        N_i} \mathbbm u_n} \sum_{n<N_i}
    (U_n(\omega)-\mathbbm u_n)G_n}\le 4(\sigma-1)\text{  for }\omega\in \Omega'_\sigma.
\end{equation}
For a given $N$, choose the $i$ so that
\begin{equation}
  \label{eq:112}
  N_i\le N<N_{i+1}.
\end{equation}
To estimate the difference between the $N$th and $N_i$th average,
denote
\begin{equation}
  \label{eq:113}
  H_n=(U_n(\omega)-\mathbbm u_n)G_n,
\end{equation}
and write
\begin{align}
  \label{eq:114}
  \frac1{\sum_{n< N} \mathbbm u_n}& \sum_{n<N} H_n-\frac1{\sum_{n<
      N_i} \mathbbm u_n}
  \sum_{n<N_i} H_n\\
  &=\paren*{\frac1{\sum_{n< N} \mathbbm u_n} \sum_{n<N}
    H_n-\frac1{\sum_{n< N} \mathbbm u_n}
    \sum_{n<N_i} H_n} +\paren*{\frac1{\sum_{n< N} \mathbbm u_n} \sum_{n<N_i} H_n-\frac1{\sum_{n< N_i} \mathbbm u_n}\sum_{n<N_i} H_n}\\
  &=\paren*{\frac1{\sum_{n< N} \mathbbm u_n} \sum_{N_i\le n<N}
    H_n}+\paren*{\sum_{n<N_i} H_n}\paren*{\frac1{\sum_{n< N} \mathbbm
      u_n}-\frac1{\sum_{n< N_i} \mathbbm u_n}}.\label{eq:115}
\end{align}
To estimate the two terms in \cref{eq:115}, we need the following
consequence of the strong law of large numbers, \Cref{thm:2},
\begin{equation}
  \label{eq:116}
  \lim_{i\to\infty} \frac{\sum_{N_i\le n<N_{i+1}} U_n(\omega)}{\sum_{N_i\le
      n<N_{i+1}} \mathbbm u_n}=1 \text{ with probability } 1.
\end{equation}
To see \cref{eq:116}, write as
\begin{equation}
  \label{eq:117}
  \frac{\sum_{N_i\le n<N_{i+1}} U_n(\omega)}{\sum_{N_i\le n<N_{i+1}} \mathbbm u_n}
  =\frac1{\sum_{N_i\le n<N_{i+1}} \mathbbm u_n} \paren*{ \sum_{ n<N_{i+1}} U_n(\omega)-\sum_{n<N_{i}} U_n(\omega)}.
\end{equation}
Writing
\begin{equation}
  \label{eq:118}
  \frac1{\sum_{N_i\le n<N_{i+1}} \mathbbm u_n} \sum_{ n<N_{i+1}} U_n(\omega)=\frac{\sum_{n<N_{i+1}} \mathbbm u_n}{\sum_{N_i\le n<N_{i+1}}
    \mathbbm u_n}\cdot \frac1{\sum_{n<N_{i+1}} \mathbbm u_n}\sum_{
    n<N_{i+1}} U_n(\omega),  
\end{equation}
we see that by the strong law of large numbers, we have, with
probability $1$,
\begin{align}
  \label{eq:119}
  \lim_{i\to\infty} \frac1{\sum_{N_i\le n<N_{i+1}} \mathbbm u_n
  }\sum_{ n<N_{i+1}} U_n(\omega)&=\lim_{i\to\infty} \frac{\sum_{n<N_{i+1}}
    \mathbbm u_n}{\sum_{N_i\le n<N_{i+1}}
    \mathbbm u_n}\\
  \intertext{using \cref{eq:106}}
  &=\lim_{i\to\infty} \frac{\sigma^{i+1}+O(1)}{\sigma^{i+1}-\sigma^{i}+O(1)}\\
  &=\frac{\sigma}{\sigma-1}\label{eq:120}.
\end{align}
Similarly
\begin{align}
  \label{eq:121}
  \lim_{i\to\infty} \frac1{\sum_{N_i\le n<N_{i+1}} \mathbbm u_n
  }\sum_{ n<N_{i}} U_n(\omega)&=\lim_{i\to\infty} \frac{\sum_{n<N_{i}}
    \mathbbm u_n}{\sum_{N_i\le n<N_{i+1}}
    \mathbbm u_n}\\
  &=\frac{1}{\sigma-1}\label{eq:122}.
\end{align}
Now take the difference of \cref{eq:120,eq:122} to get our claim in
\cref{eq:116}.

Let's go back to our main task of estimating the two terms in
\cref{eq:115}. Estimate as
\begin{align}
  \label{eq:123}
  \abs*{\frac1{\sum_{n< N} \mathbbm u_n} \sum_{N_i\le n<N} H_n}&\le
  \frac1{\sum_{n< N_i} \mathbbm u_n} \sum_{N_i\le n<N_{i+1}} \abs*{H_n}\\
  &\le \frac1{\sum_{n< N_i} \mathbbm u_n} \sum_{N_i\le n<N_{i+1}}
  (U_n(\omega)+\mathbbm u_n)\\
  &=\frac{\sum_{N_i\le n<N_{i+1}} \mathbbm u_n}{\sum_{n< N_i} \mathbbm
    u_n}\cdot \frac1{\sum_{N_i\le n<N_{i+1}} \mathbbm u_n}\sum_{N_i\le
    n<N_{i+1}} (U_n(\omega)+\mathbbm u_n).
\end{align}
Take ``$\limsup_{i\to\infty}\max_{N_i\le N<N_{i+1}}$'' of both ends,
to get
\begin{align}\label{eq:124}
  \limsup_{i\to\infty}\max_{N_i\le N<N_{i+1}}&\abs*{\frac1{\sum_{n< N}
      \mathbbm u_n} \sum_{N_i\le n<N} H_n}\\
  &\le \limsup_{i\to\infty}\frac{\sum_{N_i\le n<N_{i+1}} \mathbbm
    u_n}{\sum_{n< N_i} \mathbbm u_n}\cdot \frac1{\sum_{N_i\le
      n<N_{i+1}} \mathbbm u_n}\sum_{N_i\le n<N_{i+1}}
  (U_n(\omega)+\mathbbm u_n)\\
  \intertext{using \cref{eq:116} } &=2\cdot
  \lim_{i\to\infty}\frac{\sum_{N_i\le n<N_{i+1}} \mathbbm
    u_n}{\sum_{n< N_i} \mathbbm
    u_n}\\
  \intertext{using \cref{eq:106}} &=2\cdot \lim_{i\to\infty}
  \frac{\sigma^{i+1}-\sigma^i
    +O(1)}{\sigma^i+O(1)}\\
  &=2(\sigma -1)\label{eq:125}.
\end{align}
Similarly
\begin{equation}
  \label{eq:126}
  \limsup_{i\to\infty}\max_{N_i\le N<N_{i+1}}\abs*{\paren*{\sum_{n<N_i} H_n}\paren*{\frac1{\sum_{n< N}
        \mathbbm u_n}-\frac1{\sum_{n< N_i} \mathbbm u_n}}}\le 2(\sigma-1).
\end{equation}
The estimates in \cref{eq:126,eq:125} imply \cref{eq:111}.

The consequence of this lacunary subsequence trick is that for the
averages in \cref{eq:100}, it is enough to prove, for all $\sigma>1$,
almost sure convergence along the subsequence $ I_\sigma$. This is
accomplished if we ensure that the sequence
$\cbrace*{\epsilon_1,\epsilon_2,\dots,\epsilon_n,\dots }$ satisfies
\begin{equation}
  \label{eq:127}
  \sum_{N\in I_\sigma}\epsilon_N<\infty \text{ for all }\sigma>1
\end{equation}
in addition to the requirement in \cref{eq:81}. To see this
implication of the condition in \cref{eq:127}, note first that by the
Borel-Cantelli lemma we have, for almost every $\omega$,
\begin{equation}
  \label{eq:128}
  \sup_{(X,T,\mu),F_i}\int_X\abs*{\frac1{\sum_{n<N}\mathbbm
      u_n}\sum_{n<N}(U_n(\omega)-\mathbbm u_n) \cdot T^nF_1(x)\cdot
    T^{2n}F_2(x)\cdots T^{\ell n}F_\ell(x)}\le \epsilon_N\text{ for
  }N\ge N(\omega). 
\end{equation}
By \cref{eq:127}, we have, for
almost every $\omega$,
\begin{equation}
  \begin{aligned}
    \sup_{(X,T,\mu),F_i}\int_X&\sum_{N\in
      I_\sigma}\abs*{\frac1{\sum_{n<N}\mathbbm
        u_n}\sum_{n<N}(U_n(\omega)-\mathbbm u_n) \cdot T^nF_1(x)\cdot
      T^{2n}F_2(x)\cdots T^{\ell n}F_\ell(x)}\\
    &\le \sum_{N\in
      I_\sigma}\sup_{(X,T,\mu),F_i}\int_X\abs*{\frac1{\sum_{n<N}\mathbbm
        u_n}\sum_{n<N}(U_n(\omega)-\mathbbm u_n) \cdot T^nF_1(x)\cdot
      T^{2n}F_2(x)\cdots T^{\ell n}F_\ell(x)}\\
    &<\infty.
  \end{aligned} \label{eq:129}
\end{equation}
This implies that the series $\sum_{N\in
  I_\sigma}\abs*{\frac1{\sum_{n<N}\mathbbm
    u_n}\sum_{n<N}(U_n(\omega)-\mathbbm u_n) \cdot T^nF_1(x)\cdot
  T^{2n}F_2(x)\cdots T^{\ell n}F_\ell(x)}$ is convergent for almost
every $x$, and hence $\frac1{\sum_{n<N}\mathbbm
  u_n}\sum_{n<N}(U_n(\omega)-\mathbbm u_n) \cdot T^nF_1(x)\cdot
T^{2n}F_2(x)\cdots T^{\ell n}F_\ell(x)\to 0$ for almost every $x$ as $N\in
I_\sigma$ goes to $\infty$.

To ensure \cref{eq:127}, note that in our case, since we at least
assume $\frac{\sum_{n<N}\mathbbm u_n}{N^{1-1/(\ell +1)}}\to\infty$,
the subsequence $ I_\sigma$ does not increase faster to $\infty$ than
a geometric progression. So a near optimal choice for $\epsilon_N$ to
ensure \cref{eq:127} is $\epsilon_N=\log^{-(1+\delta)} N$ for some
fixed positive $\delta$.

Putting all this together, we get the following theorem for the case
when the $F_i$ are indicator functions

\begin{bigthm}[Random pointwise $\ell$-averaging theorem]
  \label{bigthm:4}
  Let $\ell$ be a positive integer and suppose that for some positive
  $\delta$ we have
  \begin{equation}
    \label{eq:130}
    \mathbbm u_n\cdot \frac{n^{1/(\ell +1)}}{\log^{2+\delta}n}\to \infty.
  \end{equation}

  Then, with probability $1$, in any dynamical system $(X,T)$,
  \begin{equation}
    \label{eq:131}
    \lim_{N\to \infty}\frac1{\sum_{n< N} U_n(\omega)}\sum_{n< N} (U_n(\omega)-\mathbbm u_n)\cdot T^{
      n}F_{1}(x)\cdot T^{2n}F_{2}(x)\cdots
    T^{ \ell n}F_{\ell}(x)=0\text{ for a.e. }x\in X,
  \end{equation}
  for bounded functions $F_i$.
\end{bigthm}
The main utility of this theorem is that once we have almost
everywhere convergence for the multiple averages
\begin{equation}
  \label{eq:132}
  \frac1{N}\sum_{n< N} T^{n}F_{1}(x)\cdot T^{2n}F_{2}(x)\cdots
  T^{ \ell n}F_{\ell}(x), 
\end{equation}
then the averages
\begin{equation}
  \label{eq:133}
  \frac1{\sum_{n< N} U_n}\sum_{n< N} U_n\cdot T^{n}F_{1}(x)\cdot T^{2n}F_{2}(x)\cdots
  T^{ \ell n}F_{\ell}(x)
\end{equation}
along the random sequence also converge and to the same limit for
almost every $x$.  So far, the a.e. convergence of \cref{eq:132} in
\emph{every} dynamical system is known only in case $\ell=1,2$.  For
$\ell=1$, this is the pointwise ergodic theorem, for $\ell=2$ it was
proved by \cite{MR1037434}. Moreover, the pointwise convergence of the
averages in \cref{eq:132} for any $\ell$ is known in a great variety
of particular dynamical systems such as in
\begin{itemize}
\item discrete spectrum or quasi-discrete spectrum dynamical systems;
\item K-systems (cf. \cite{MR1422310});
\item a large class of skew-product extensions of ergodic compact
  group translations (cf. \cite{MR1242635});
\item a lot of dynamical systems with a compact topological structure
  where \emph{everywhere} convergence can be proved for continuous
  functions, hence the general pointwise ergodic theorem follows by a
  maximal inequality.  An example for this that we haven't mentioned
  is a horocycle flow on a compact surface with negative curvature,
  or, more generally, unipotent transformations (Ratner's theory).
\end{itemize}

So far, we showed how to prove \Cref{bigthm:4} for indicator
functions. The extension to bounded functions is more or less
standard: Since we have a.e. convergence for a dense set of functions,
namely for linear combinations of indicators, we just need a ``maximal
inequality'' argument.  As we will see, in this maximal inequality
argument, we will work with a set of $\omega$ which satisfy conditions
that are \emph{independent} of the dynamical system.  The two main
conditions are  the a.e. convergence of the averages in \cref{eq:131}
when the $F_i$ are indicators, and the maximal inequality in
\cref{eq:137} below. We will also use the strong law of large numbers,
\Cref{thm:2}. 

The way we'll do the extension of a.e. convergence to bounded
functions is that we first extend a.e. convergence to all bounded
$F_1$, then we extend it to all bounded $F_1,F_2$, then to all bounded
$F_1,F_2,F_3$, etc.

For this argument, it is convenient to introduce three notations: one
for our averages, normalized by $\sum_{n< N} \mathbbm u_n$ instead of
$\sum_{n< N} U_n$
\begin{equation}
  \label{eq:134}
  A_N(F_1,F_2,\dots,F_\ell)=\frac1{\sum_{n< N} \mathbbm u_n}\sum_{n< N} (U_n-\mathbbm u_n)\cdot T^{
    n}F_{1}\cdot T^{2n}F_{2}\cdots
  T^{ \ell n}F_{\ell},
\end{equation}
and two for the 1-linear, dominating averages
\begin{equation}
  \label{eq:135}
  B_NF=\frac1{\sum_{n< N} \mathbbm u_n}\sum_{n< N} \mathbbm u_n\cdot
  \abs*{T^{n}F},
  \qquad
  C_NF=\frac1{\sum_{n< N} \mathbbm u_n}\sum_{n< N} U_n\cdot
  \abs*{T^{n}F}.
\end{equation}
A summation by parts argument shows that $B_N$ is pointwise dominated
by the usual ergodic averages $\frac1{N}\sum_{n< N} \abs*{T^{n}F}$,
and hence we have the maximal inequality for the $B_N$
\begin{equation}
  \label{eq:136}
  \int_X\paren*{\sup_N B_NF}^2\le c\int_XF^2.
\end{equation}
With probability $1$, we also have a maximal inequality for the $C_N$
\begin{equation}
  \label{eq:137}
  \int_X\paren*{\sup_N C_NF}^2\le c_\omega\int_XF^2.
\end{equation}
This was established by \cite{MR937581} for a much wider range of
$\mathbbm u_n$ than we consider in this theorem: they just need to
satisfy
\begin{equation}
  \label{eq:138}
  \frac{\sum_{n< N} \mathbbm u_n}{(\log\log N)^{1+\delta}\log N} \to
  \infty\text{ for some positive }\delta.
\end{equation}
That this extra $\log\log N$ factor is \emph{necessary} for the
maximal inequality and a.e. convergence, compared to the lone $\log N$
needed for mean convergence, was proved in \cite{MR1721622}.

Let now $F_1$ be a function bounded by $1$, and assume
$F_2,F_3,\dots,F_\ell$ are indicators.  We want to show that, with
probability $1$, the oscillation of the $A_N$ is $0$ almost surely
\begin{equation}
  \label{eq:139}
  \mu\paren*{\limsup_N \abs*{A_N(F_1,F_2,\dots,F_\ell) }=0}=1.
\end{equation}
We show this by showing that for any $\epsilon>0$
\begin{equation}
  \label{eq:140}
  \mu\paren*{\limsup_N \abs*{A_N(F_1,F_2,\dots,F_\ell) }>\epsilon}\le \epsilon.
\end{equation}
Let $\epsilon>0$ be given and let $G_1$ be a linear combination of
indicators so that
\begin{equation}
  \label{eq:141}
  \paren*{\int_X \abs*{F_1-G_1}^2}^{1/2}<\delta,
\end{equation}
where we'll choose $\delta$ later appropriately for $\epsilon$. By
\Cref{bigthm:4}, we have
\begin{equation}
  \label{eq:142}
  \limsup_N \abs*{A_N(G_1,F_2,\dots,F_\ell)(x) }=0 \text{ for a.e. }x.
\end{equation}
We then have
\begin{equation}
  \label{eq:143}
  \limsup_N \abs*{A_N(F_1,F_2,\dots,F_\ell)(x) }= \limsup_N \abs*{A_N(F_1-G_1,F_2,\dots,F_\ell)(x) }\text{ for a.e. }x.
\end{equation}
Using the boundedness of the $F_i$ by $1$ we can estimate pointwise as
\begin{equation}
  \label{eq:144}
  \abs*{A_N(F_1-G_1,F_2,\dots,F_\ell)(x) }
  \le B_N\paren*{F_1-G_1}(x)+
  C_N\paren*{F_1-G_1}(x).
\end{equation}
We thus have
\begin{align}
  \label{eq:145}
  \int_X \limsup_N \abs*{A_N(F_1-G_1,F_2,\dots,F_\ell) }
  &\le \int_X \sup_N B_N\paren*{F_1-G_1}+\int_X \sup_N C_N\paren*{F_1-G_1} \\
  \intertext{by the Cauchy-Schwarz-Bunyakovsky inequality for the
    integral $\int_X$}
  &\le \paren*{\int_X \paren*{\sup_N B_N\paren*{F_1-G_1}}^2}^{1/2}+\paren*{\int_X \paren*{\sup_N C_N\paren*{F_1-G_1}}^2}^{1/2}\\
  \intertext{using the maximal inequalities for $B_N$ in \cref{eq:136}
    and for $C_N$ in \cref{eq:137}}
  &\le \paren*{c\int_X \paren*{F_1-G_1}^2}^{1/2}+\paren*{c_\omega\int_X \paren*{F_1-G_1}^2}^{1/2}
  \intertext{by the approximation in \cref{eq:141}} &\le c_\omega
  \delta.
\end{align}
Using this and Markov's inequality, we finally get
\begin{align}
  \label{eq:146}
  \mu\paren*{\limsup_N \abs*{A_N(F_1,F_2,\dots,F_\ell) }>\epsilon}&\le
  \frac1\epsilon \cdot \int_X \limsup_N \abs*{A_N(F_1-G_1,F_2,\dots,F_\ell) }\\
  &\le \frac{c_\omega\delta}\epsilon.
\end{align}
We see, we just have to choose $\delta$ small enough so that
\begin{equation}
  \label{eq:147}
  \frac{c_\omega\delta}\epsilon<\epsilon
\end{equation}
to get \cref{eq:140}.

What we have accomplished so far is that our averages
$A_N(F_1,F_2,\dots,F_\ell)$ converge a.e. to $0$ for all bounded $F_1$
and indicators $F_2,F_3,\dots,F_\ell$.  But we can repeat the argument
with $F_2$ in place of $F_1$ to get that $A_N(F_1,F_2,\dots,F_\ell)$
converge a.e. to $0$ for all bounded $F_1,F_2$ and indicators
$F_3,F_4,\dots,F_\ell$.  Repeating the argument $\ell$ times, we
eventually get that $A_N(F_1,F_2,\dots,F_\ell)$ converge a.e. to $0$
for bounded $F_1,F_2,\dots,F_\ell$.

\section{Proof of \Cref{bigthm:3}, the semirandom convergence theorem
}
\label{sec:proof--1}

Our first observation is that the a.e. convergence of our averages
\begin{equation}
  \label{eq:148}
  \frac1N\sum_{n<N}T^nF_1\cdot(x)T^{r_n}F_2(x)
\end{equation}
to $\overline F_1\overline F_2$ needs to be proved only for indicator
functions $F_1,F_2$. This follows from a similar maximal inequality
argument as we used to extend \Cref{bigthm:4} from indicators to
bounded functions.

So, unless we say otherwise, we now assume, for the rest of the proof,
that the $F_i$ are indicator functions. We will follow the general
structure of the proof in \cite{Fra}.  As in the proof of
\Cref{bigthm:4}, it is enough to prove convergence along a sparse
subsequence of the $N$. As in that proof, for $\sigma>1$, we choose
the subsequence $I_\sigma=\cbrace{N_1<N_2<N_3<\dots}$ so that
$\sum_{n<N_i}\mathbbm u_n\approx \sigma^i$ (and hence
$\sum_{n<N_i}U_n\approx \sigma^i$)
\begin{equation}
  \label{eq:149}
  \sum_{n<N_i} \mathbbm u_n \ge \sigma^i \text{ but }
  \sum_{n<N} \mathbbm u_n < \sigma^i \text{ for }N<N_i.
\end{equation}
To be exact, since $ \mathbbm u_n \to 0 $,  we have
\begin{equation}
  \label{eq:150}
  \sum_{n<N_i}\mathbbm
  u_n= \sigma^i+o(1)\text{ for all }i=1,2,3,\dots
\end{equation}

We notice, that the averages in \cref{eq:148} can be written in the
form
\begin{equation}
  \label{eq:151}
  \frac1{\sum_{n<N}U_n}\sum_{n<N}U_nT^{U_1+U_2+\dots+U_{n}}F_1T^nF_2 .
\end{equation}
The two main steps in proving that the above converges a.e. are that
the ``partial'' expectations
\begin{equation}
  \label{eq:152}
  \frac1{\sum_{n<N}\mathbbm u_n}\sum_{n<N}\mathbbm u_nT^{U_1+U_2+\dots+U_{n}}F_1T^nF_2 
\end{equation}
converge to $\overline F_1\overline F_2$, and that the difference
between the two averages
\begin{equation}
  \label{eq:153}
  \frac1{\sum_{n<N}\mathbbm u_n}\sum_{n<N}(U_n-\mathbbm u_n)T^{U_1+U_2+\dots+U_{n}}F_1T^nF_2 
\end{equation}
converges to $0$.

We will be able to prove the convergence of the partial expectations
in \cref{eq:152} for \emph{any} sequence $(\mathbbm u_n)$ with
$\mathbbm u_n \to 0$ and satisfying the dissipative property
$\sum_{n}\mathbbm u_n=\infty$. To emphasize this fact, we formulate
the result explicitly
\begin{thm}\label{thm:1}
  Suppose the decreasing sequence $\cbrace{\mathbbm u_1\ge\mathbbm
    u_3\ge \mathbbm u_1\ge\dots}$ satisfies
  \begin{equation}
    \label{eq:154}
    \sum_{n}\mathbbm u_n=\infty,
  \end{equation}
  and
  \begin{equation}
    \label{eq:155}
    \lim_{n\to\infty}\mathbbm u_n=0.
  \end{equation}

  We then have, with probability $1$, in any dynamical system,
  \begin{equation}
    \label{eq:156}
    \lim_{N\to\infty} \frac1{\sum_{n<N}\mathbbm
      u_n}\sum_{n<N}\mathbbm u_nT^{U_1+U_2+\dots+U_{n}}F_1T^nF_2
    =\overline F_1\overline F_2\text{ for a.e.}
    x\in X\text{ and any bounded functions } F_i.
  \end{equation}
\end{thm}
\begin{proof}
  We divide the proof into two cases: first, when $F_2$ is
  $T$-invariant, and second, when $F_2$ is orthogonal to all
  $T$-invariant functions. The function $F_1$ is an indicator
  throughout the proof.

  Let thus $F_2$ be $T$-invariant. We then need to prove the pointwise
  convergence of
  \begin{equation}
    \label{eq:157}
    \frac1{\sum_{n<N}\mathbbm u_n}\sum_{n<N}\mathbbm
    u_nT^{U_1+U_2+\dots+U_{n}}F_1.
  \end{equation}
  The a.e. convergence of the averages
  \begin{equation}\label{eq:158}
    \frac1{\sum_{n<N}U_n}\sum_{n<N}U_nT^{U_1+U_2+\dots+U_{n}}F_1
  \end{equation}
  to $\overline F_1$ follows from the pointwise ergodic theorem. To
  see this, let $r_1<r_2<\dots <r_k<\dots <r_K<N$ be all the elements
  of our random sequence which are less than $N$. Since
  $U_1,U_2,\dots$ is the indicator of the random sequence, $U_n=1$
  exactly when $n=r_k$ for some $k\le K$.  It follows that
  \begin{equation}
    \label{eq:159}
    \sum_{n<N}U_n=K,
  \end{equation}
  and
  \begin{equation}
    \label{eq:160}
    U_nT^{U_1+U_2+\dots+U_{n}}=
    \begin{dcases*}
      T^k& if $n=r_k$ for some $k\le K$,\\
      0& otherwise,
    \end{dcases*}
  \end{equation}
  so the average in \cref{eq:158} is the $K$th ergodic average
  $\frac1K\sum_{1\le k\le K} T^kF_1$.

  Since the averages in \cref{eq:158} converge a.e. to $\overline
  F_1$, by the strong law of large numbers, the averages
  \begin{equation}
    \label{eq:161}
    \frac1{\sum_{n<N}\mathbbm u_n}\sum_{n<N}U_nT^{U_1+U_2+\dots+U_{n}}F_1.
  \end{equation}
  also converge to $\overline F_1$. We then need to show that, with
  full probability, the difference between the two averages converges
  to $0$ a.e.
  \begin{equation}
    \label{eq:162}
    \lim_{N\to \infty} \sup_{(X,\mu,T),F_1} \abs*{\frac1{\sum_{n<N}\mathbbm
        u_n}\sum_{n<N}(U_n-\mathbbm u_n)T^{U_1+U_2+\dots+U_{n}}F_1}=0
    \text{ with probability }1.
  \end{equation}
  Let $\epsilon_N\to 0$. Later we'll find that the fastest rate our
  method allows $\epsilon_N$ to go to $0$ is
  $O\paren*{\paren*{\sum_{n<N}\mathbbm u_n}^{-1/4}}$ which is plenty,
  as we will soon see. We shall prove
  \begin{equation}
    \label{eq:163}
    P\,\paren*{
      \sup_{(X,\mu,T),F_1}\int_X\abs*{\frac1{\sum_{n<N}\mathbbm
          u_n}\sum_{n<N}(U_n-\mathbbm u_n)T^{U_1+U_2+\dots+U_{n}}F_1}>
      \epsilon_N}< \exp{\paren*{-c
        \epsilon_N^2\sum_{n<N}\mathbbm u_n}}
  \end{equation}
  for all large enough $N$.  This implies \cref{eq:162} if
  $\epsilon_N$ can be chosen so that, for all $\sigma>1$ we have
  \begin{align}
    \sum_{N\in I_\sigma}\epsilon_N&<\infty,\label{eq:164}\\
    \intertext{and} \sum_{N\in I_\sigma}\exp{\paren*{-c
        \epsilon_N^2\sum_{n<N}\mathbbm u_n}}&<\infty,\label{eq:165}
  \end{align}
  where $I_\sigma$ is defined in \cref{eq:149}. This is because we
  need to show convergence only for the subsequence $I_\sigma$.

  Following the proof of \Cref{bigthm:2}, a duality and a transference
  argument shows that \cref{eq:163} follows from
  \begin{multline}
    \label{eq:166}
    P\,\paren*{\sup_{f\in \mathcal{F},g\in
        \mathcal{G}}\frac1{\sum_{n<N}\mathbbm
        u_n}\sum_{n<N}(U_n-\mathbbm u_n)
      \frac1Q\sum_{a<Q}g(a)f\paren[\big]{a+\paren*{U_1+U_2+\dots+U_{n}}}
      >\epsilon_N}\\ < \exp{\paren*{-c
        \epsilon_N^2\sum_{n<N}\mathbbm u_n}} \text{ for all large
      enough }N,
  \end{multline}
  where $\mathcal{G}$ is the same as before, that is, it contains $\pm
  1$ valued functions supported on the interval $[0,Q)$
  \begin{equation}
    \label{eq:167}
    \mathcal{G}=\cbrace[\Big]{g\mid g:[0,Q)\to \cbrace*{-1,1}},
  \end{equation}
  and $\mathcal{F}$ contains indicators supported on the interval
  $\Bigl [0,Q+U_1+U_2+\dots+U_{N}\Bigr)$
  \begin{equation}
    \label{eq:168}
    \mathcal{F}=\cbrace*{f\mid f:\Bigl [0,Q+U_1+U_2+\dots+U_{N}\Bigr)\to
      \cbrace*{0,1}}.
  \end{equation}
  In fact, except on a set with small probability, we can replace the
  interval $[0,Q+U_1+U_2+\dots+U_{N}\Bigr)$ above, which depends on
  $\omega$, by the interval $\Bigl [0,Q+2\sum_{n<N}\mathbbm
  u_n\Bigr)$. This is because when we put $X_n=U_n-\mathbbm u_n$ and
  $t=\sum_{n<N}\mathbbm u_n$ in Bernstein's inequality, \Cref{lem:1},
  we get
  \begin{equation}
    \label{eq:169}
    P\,\paren*{U_1+U_2+\dots+U_{N}-\sum_{n<N}\mathbbm
      u_n>\sum_{n<N}\mathbbm u_n}< \exp{\paren*{-c\sum_{n<N}\mathbbm u_n}},
  \end{equation}
  and, of course, $\sum_{N\in
    I_\sigma}\exp{\paren*{-c\sum_{n<N}\mathbbm u_n}}<\infty$. So from
  now on, $\mathcal{F}$ denotes indicators on the interval $\Bigl
  [0,Q+2\sum_{n<N}\mathbbm u_n\Bigr)$
  \begin{equation}
    \label{eq:170}
    \mathcal{F}=\cbrace*{f\mid f:\Bigl [0,Q+2\sum_{n<N}\mathbbm u_n\Bigr)\to
      \cbrace*{0,1}}.
  \end{equation}

  The next step in our method is to divide up the interval $\Bigl
  [0,Q+2\sum_{n<N}\mathbbm u_n\Bigr)$ into residue classes $\pmod{q}$
  for an appropriate $q$. The same way as in previous proofs, we get
  that a good choice for $q$ is $q=Q$. We also get the requirement,
  using the cardinality of functions supported on a single residue
  class $\pmod q$ in $\Bigl [0,Q+2\sum_{n<N}\mathbbm u_n\Bigr)$ and
  the cardinality of $\mathcal{G}$, that
  \begin{equation}
    \label{eq:171}
    \frac{2\sum_{n<N}\mathbbm u_n}{q}+Q\le c\epsilon_N^2\sum_{n<N}\mathbbm u_n.
  \end{equation}
  Since $q=Q$, the left hand side of \cref{eq:171} is smallest when
  \begin{equation}
    \label{eq:172}
    q=Q=c\paren*{\sum_{n<N}\mathbbm u_n}^{1/2}.
  \end{equation}
  In the above, nothing changes significantly, if we choose the
  constant $c$ so that $Q$ is an integer. Substituting \cref{eq:172}
  into \cref{eq:171}, we get
  \begin{equation}
    \label{eq:173}
    \epsilon_N=c\paren*{\sum_{n<N}\mathbbm u_n}^{-1/4}.
  \end{equation}
  With these values all our requirements are satisfied including
  \cref{eq:164,eq:165}, and our proof of \cref{eq:162} is finished.

  Our next step in proving \Cref{thm:1} is to consider the case when
  $F_2$ is orthogonal to the invariant functions.  Since in this case,
  $\overline F_2=0$, we want to prove
  \begin{equation}
    \label{eq:174}
    \frac1{\sum_{n<N}\mathbbm
      u_n}\sum_{n<N}\mathbbm
    u_nT^{U_1+U_2+\dots+U_{n}}F_1T^nF_2\to 0.
  \end{equation}
  By a maximal inequality argument---similar but simpler than the one
  we used to extend \Cref{bigthm:4} to bounded functions---it is
  enough to prove \cref{eq:172} for a class of functions which
  generates an $L^2$-dense linear subspace of the orthocomplement of
  the $T$-invariant functions. This class of functions will be the
  ``bounded'' coboundaries, so those functions $F_2$ which can be
  written in the form $F_2=TG-G$ for some bounded $G$.

  So let $F_2$ be a coboundary. We shall prove a much more general
  result: The only property of the sequence $\cbrace*{
    T^1F_2(x),T^2F_2(x),\dots}$ we use is that
  $\frac1N\sum_{L<n<L+N}T^nF_2(x)$ converges to $0$ \emph{uniformly}
  in $L$ as $N\to\infty$. The only property of the random sequence
  $\cbrace*{r_1<r_2<r_3<\dots}$ we use is that it has $0$ density as a
  consequence of the assumption that $\mathbbm u_n\to 0$.
  \begin{lem}
    \label{lem:2}
    Suppose the set $R$ of positive integers has 0 density, and let
    $X(n)$ denote its indicator. Let $\cbrace*{g_1,g_2,g_3,\dots}$ be
    a bounded sequence of real numbers which satisfy
    \begin{equation}
      \label{eq:175}
       \lim_{L\to\infty}\frac1L \sum_{M\le n <
        M+L }g_n = 0\text{\quad uniformly in }M.
    \end{equation}
  
    Then for any bounded sequence $f_1,f_2,f_3\dots$ of numbers and
    decreasing sequence $w_1\ge w_2\ge w_3\ge \dots$ of weights with
    $\sum_n w_n=\infty$ we have
    \begin{equation}
      \label{eq:176}
      \lim_{N\to\infty}\frac1{\sum_{n<N}w_n}\sum_{n< N} w_nf_{X(1)+X(2)+\dots+X(n)}g_n = 0.
    \end{equation}
  \end{lem}
  \begin{proof}
    A summation by parts argument shows that we can assume all the
    weights $w_n$ are equal to $1$, so we need to show
    \begin{equation}
      \label{eq:177}
      \frac1{N}\sum_{n< N} f_{X(1)+X(2)+\dots+X(n)}g_n\to 0.
    \end{equation}
    The main observation is that the complement of $R$ is
    \emph{basically} a set which is the union of longer and longer
    intervals. More precisely, for any given length $L$, there is a
    $S\subset R^c$ of density $1$ so that $S$ is the union of
    intervals of length at least $L$. To see this, suppose to the
    contrary: there is a length $L$ so that the set of $k$,
    $k=1,2,3,\dots$, for which the interval $[kL,(k+1)L)$ contains a
    point from $R$ has positive upper density, say, $\delta >0$.  But
    then the upper density of $R$ is at least $\delta/L$, a
    contradiction.

    To prove the lemma, let $L$ be arbitrary, and consider the average
    \begin{equation}
      \frac1{KL} \sum_{k\le K}\sum_{kL\le n<(k+1)L} f_{X(1)+X(2)+\dots+X(n)}g_n.
    \end{equation}
    Let $k_1<k_2<\dots$ be a sequence so that $R\subset \cup_i
    [k_iL,(k_i+1)L)$, and the density of $(k_i)$ is 0. Write
    \begin{equation}
    \begin{multlined}
      \frac1{KL} \sum_{k\le K}\sum_{kL\le n <(k+1)L }  f_{X(1)+X(2)+\dots+X(n)}g_n\\
      =\frac1{KL} \sum_{i}\sum_{k_iL\le n<(k_i+1)L}
      f_{X(1)+X(2)+\dots+X(n)}g_n + \frac1{KL} \sum_{k\ne
        k_i}\sum_{kL\le n<(k+1)L} f_{X(1)+X(2)+\dots+X(n)}g_n.
    \end{multlined}\label{eq:178}
  \end{equation}

    The first sum above goes to 0, since $(k_i)$ has 0 density and the
    functions are assumed to be bounded.  For the second sum, note
    that on each interval $ [kL,(k+1)L)$, we have $X(n)=0$, hence
    $f_{X(1)+X(2)+\dots+X(n)}$ is constant.  We can estimate
    \begin{equation}
      \label{eq:179}
      \abs*{\frac1{KL} \sum_{k\ne k_i}\sum_{kL\le n<(k+1)L}
        f_{X(1)+X(2)+\dots+X(n)}g_n}\le \frac1{K}\sum_{k\ne k_i}\abs*{\frac1{L} \sum_{kL\le n<(k+1)L}g_n}.
    \end{equation}
    By the assumption in \cref{eq:175}, $\abs*{\frac1{L} \sum_{kL\le
        n<(k+1)L}g_n}$ is uniformly small if $L$ is big, finishing the
    proof of the lemma.
  \end{proof}
  This ends the proof of \Cref{thm:1}.
\end{proof}
Our last task in the proof of \Cref{bigthm:3} is to prove that with
probability $1$ we have, in every dynamical system,
\begin{equation}
  \label{eq:180}
  \lim_{N\to\infty}\frac1{\sum_{n<N}\mathbbm u_n}\sum_{n<N}(U_n-\mathbbm u_n)T^{U_1+U_2+\dots+U_{n}}F_1(x)T^nF_2(x)
  = 0\text{ for almost every }x. 
\end{equation}
Let $\delta>0$ and
\begin{equation}
  \label{eq:181}
  \epsilon_N = \paren*{\log\sum_{n<N}\mathbbm
    u_n}^{-(1+\delta)}.
\end{equation}
As many times before in this paper, our task is reduced to proving an
inequality for the probability of averages on the integers
\begin{equation}
\begin{multlined} 
  P\,\paren*{\sup_{f_i\in \mathcal F_i, g\in \mathcal{G}}
    \frac1{\sum_{n<N}\mathbbm u_n}\sum_{n<N}(U_n-\mathbbm
    u_n)\frac{1}{Q}\sum_{a<Q}g(a)f_1\paren[\Big]{a+\paren[\big]{U_1+U_2+\dots+U_{n}}}f_2(a+n)
    >\epsilon_N}\\
  \le \exp{\paren*{ -c\epsilon_N^2 \sum_{n<N}\mathbbm u_n}},
  \text{ for all large enough }N,
\end{multlined}\label{eq:182}
\end{equation}
where $Q$ is to be determined momentarily, $\mathcal{F}_1$ contains
all the indicators supported in the interval\footnote{We already
  showed in \cref{eq:169} that we make only a small error in
  probability, if we replace the random interval
  $\Bigl[0,Q+\sum_{n<N}U_n\Bigr)$ by the deterministic one $\Bigl[0,Q+2\sum_{n<N}\mathbbm
  u_n\Bigr)$.}  $\Bigl[0,Q+2\sum_{n<N}\mathbbm u_n\Bigr)$
\begin{align}
  \mathcal{F}_1&=\cbrace*{f\mid f:\Bigl[0,Q+2\sum_{n<N}\mathbbm
    u_n\Bigr)\to \cbrace*{0,1}}, \intertext{$\mathcal{F}_2$ contains
    all the indicators supported in the interval $[0,Q+N)$,}
  \mathcal{F}_2&=\cbrace[\Big]{f\mid f:[0,Q+N)\to \cbrace*{0,1}},
  \intertext{and $\mathcal{G}$ contains all the $\pm 1$ valued
    functions supported on the interval $[0,Q)$}
  \mathcal{G}&=\cbrace[\Big]{g\mid g:[0,Q)\to \cbrace*{-1,1}}.
\end{align}

We now choose the relatively prime moduli $q_1,q_2$, and we divide up
the support of $f_i$ according to the residue classes $\pmod{q_i}$.
With the assumption that
\begin{equation}
  \label{eq:183}
  q_1q_2= Q,
\end{equation}
we further reduce the problem to the case when the $f_i$ are supported
on a fixed $r_i$ residue class $\pmod{q_i}$.  This way, following the
proof of \Cref{bigthm:1}, we get the estimate
\begin{multline}
  \label{eq:184}
  P\,\paren*{\sup_{f_i\in \mathcal F_i, g\in \mathcal{G}}
    \frac1{\sum_{n<N}\mathbbm u_n}\sum_{n<N}(U_n-\mathbbm
    u_n)\frac{1}{Q}\sum_{a<Q}g(a)
    f_1\paren[\Big]{a+\paren[\big]{U_1+U_2+\dots+U_{n}}}f_2(a+n)
    >\epsilon_N}\\
  \le q_2q_1\exp{\paren*{ Q+ \frac{\sum_{n<N}\mathbbm
        u_n}{q_1}+\frac{N}{q_2}-\frac14\epsilon_N^2 \sum_{n<N}\mathbbm
      u_n}}.
\end{multline}
Ignoring the negligible $q_1q_2<\exp{\paren*{c\log N}}$ product
(negligible compared to $\exp{\paren*{\epsilon_N^2 \sum_{n<N}\mathbbm
    u_n}}$, which is assumed to be $\exp{\paren*{O(N^{1/2})}}$), we
see our task is to choose the parameters $Q,q_1,q_2$ so that
\begin{equation}
  \label{eq:185}
  Q+ \frac{\sum_{n<N}\mathbbm
    u_n}{q_1}+\frac{N}{q_2}< c\epsilon_N^2 \sum_{n<N}\mathbbm u_n.
\end{equation}
If $Q$ and $q_i$ were real variables, the left hand side is minimal if
\begin{align}
  \label{eq:186}
  Q&= c\epsilon_N^2 \sum_{n<N}\mathbbm u_n,\\
  q_1&=c\epsilon_N^{-2},\\
  q_2&=c\frac{N}{\epsilon_N^2 \sum_{n<N}\mathbbm u_n}.
\end{align}
We certainly can choose integer values for $Q,q_i$ so that they are
within constant multiples of the above optimal values, $q_1$ and $q_2$
are coprimes, and $q_1q_2= Q$.  Using that $q_1q_2= Q$, we get the
requirement
\begin{equation}
  \label{eq:187}
  \frac{N}{\epsilon_N^4 \sum_{n<N}\mathbbm u_n}\le c \epsilon_N^2 \sum_{n<N}\mathbbm u_n,
\end{equation}
hence
\begin{equation}
  \label{eq:188}
  N^{1/2}\le c\epsilon_N^3 \sum_{n<N}\mathbbm u_n,
\end{equation}
which is possible, since the value of $\epsilon_N$ in \cref{eq:181} is
compatible with the assumption in \cref{eq:19}.  With this, our proof
is finished.

% \section{Proof cannot be improved to get $b<1/\ell$}
% \label{sec:proof-cannot-be}

% We may not have to use Massart.

\section{Notes}
\label{sec:notes}

\subsection{Finitary version of \Cref{bigthm:1}}
\label{sec:finit-vers-crefb}

We use the setup commonly used in combinatorics: for a given $N$, we
want to select each integer $1,2,\dots,N$ with a uniform probability
$\mathbbm u_N$. So for each $N$, we take $U_{N1},U_{N2},\dots,U_{NN}$,
a sequence of independent, identically distributed, $0-1$-valued
random variables, with the law
\begin{equation}
  \label{eq:189}
  P(U_{Nn}=1)=1-P(U_{Nn}=0)= \mathbbm u_N,\quad n=1,2,\dots,N.
\end{equation}
The sequence $U_{N1},U_{N2},\dots,U_{NN}$ is the indicator of the set
$R^\omega_N$,
\begin{equation}
  \label{eq:190}
  R^\omega_N=\cbrace[\big]{ n| U_{Nn}(\omega)=1}.
\end{equation}
The proof of \Cref{bigthm:1} can easily be adjusted to give
\begin{thm*}[Finitary version of \Cref{bigthm:1}]
  Let $\ell$ be a positive integer and assume
  \begin{equation}
    \label{eq:191}
    \mathbbm u_N\cdot N^{1/(\ell +1)}\to\infty.
  \end{equation}

  Then there is a set $\Omega'\subset \Omega$ with $P(\Omega')=1$ such
  that for every $\omega\in\Omega'$ the following is true: for all
  $\alpha>0$ there is $N(\alpha)$ so that if $A\subset [1,N]$ for some
  $N>N(\alpha)$ and $A$ has more than $\alpha N$ elements, then $A$
  contains an arithmetic progression $a,a+r,a+2r,\dots,a+\ell r$ of
  length $(\ell +1)$ with some $r\in R^\omega_N$. 
\end{thm*}

\subsection{Our method for bounded functions}
\label{sec:bounded-functions}

In our paper, we chose to extend almost everywhere convergence results
from indicators to bounded functions using the method of maximal
inequalities. Our method can be applied directly to bounded functions,
though. The modification in the proof is needed only in the definition
of the sets $\mathcal{F}$.  For example, consider the proof of
\Cref{bigthm:4} in the $\ell=1$ case. We need to consider the averages
\begin{equation}
  \label{eq:192}
  \frac1{\sum_{n<N}\mathbbm u_n }\sum_{n<N}\paren*{U_n -\mathbbm u_n} \frac1Q\sum_{a<Q}f(a+n)
\end{equation}
for $f:[0,Q+N)\to[-1,1]$.  As a first step, we want to replace $f$ by
a function which take on values from a sufficiently dense
\emph{discrete subset} of the interval $[0,1]$. Assume, for notational
simplicity, that $\epsilon_N$ is the reciprocal of a positive
integer. A dense enough set is the set
\begin{equation}
  \label{eq:193}
  D=\cbrace*{-1, -\paren*{\epsilon_N^{-1}
      -1}\epsilon_N,\dots -2\epsilon_N, -\epsilon_N,0,\epsilon_N,2\epsilon_N,\dots, \paren*{\epsilon_N^{-1}
      -1}\epsilon_N, 1},
\end{equation}
so we define
\begin{equation}
  \label{eq:194}
  \mathcal{F}=\cbrace[\Big]{f\mid f:[0,Q+N)\to  D}.
\end{equation}
Now, if we replace $f$ in \cref{eq:192} by a function from
$\mathcal{F}$, then with probability very close to $1$, we only make a
$2\epsilon_N$ error.  It follows that it is enough to consider
functions from $\mathcal{F}$ in \cref{eq:192}.  The cardinality of
$\mathcal{F}$ is not much bigger than it was before: since
$|D|=c\epsilon_N^{-1}$, we have
\begin{equation}
  \label{eq:195}
  |\mathcal{F}|< |D|^{2N}< \exp{\paren*{cN\log \epsilon_N^{-1}}}.
\end{equation}
Since in \Cref{bigthm:4}, $ \epsilon_N$ is a negative power of $\log
N$, this extra $\log \epsilon_N^{-1}$ factor would account for an
extra $\log\log N$ speed---negligible as the speed in \Cref{bigthm:4}
is stated, but it would not be negligible if we wanted to state the
best possible speed our proof could give.  For example, it's clear
from the proof that we can weaken the assumption in \cref{eq:130} of
\Cref{bigthm:4} to
\begin{equation}
  \label{eq:196}
  \mathbbm u_n\cdot \frac{n^{1/2}}{(\log n)^{2}\log\log ^{2+\delta}n}\to
  \infty\text{ for some } \delta>0.
\end{equation}
But if we used our method directly to bounded functions instead of
maximal inequalities we would have needed the stronger
\begin{equation}
  \label{eq:197}
  \mathbbm u_n\cdot \frac{n^{1/2}}{(\log n)^{2}\log\log ^{3+\delta}n}\to
  \infty\text{ for some } \delta>0
\end{equation}
assumption.

This is the reason why we decided to use the well-known maximal
inequality techniques, and hence we avoided the slight inefficiency
and complications in our method when dealing with bounded functions
instead of indicators.

\subsection{Convexity methods}
\label{sec:convexity-methods}

We can also use the more geometric \emph{convexity} methods to extend
almost sure convergence results from indicators to bounded functions.
This method is signified by the Krein-Milman theorem, though we need
only the simplest version of it for finite sets.  Let us briefly give
here the representation needed to do the extension from indicators to
nonnegative functions bounded by 1\footnote{Similar representation
  applies for extending results for $\pm 1$-valued functions to
  functions bounded by $1$.}. Let
\begin{align}
  \label{eq:198}
  \Gamma&=\cbrace[\Big]{g|g:[1,K]\to \cbrace*{0,1}}=\cbrace*{0,1}^K,\\
  \mathcal{F}&=\cbrace[\Big]{f|f:[1,K]\to [0,1]}=[0,1]^K.
\end{align}
We want to show that for each $f\in \mathcal{F}$ there is a
probability measure $\mu_f$ on $\Gamma$ so that
\begin{equation}
  \label{eq:199}
  f(b)=\int_\Gamma g(b) d\mu_f(g).
\end{equation}
The requirement of the representation uniquely determines the
measure. For any $b$, \cref{eq:199} looks like
\begin{equation}
  \label{eq:200}
  f(b)=1\cdot \mu_f\cbrace[\big]{g|g(b)=1}.
\end{equation}
It follows
\begin{align}
  \label{eq:201}
  \mu_f\cbrace[\big]{g|g(b)=1}&=f(b),\\
  \mu_f\cbrace[\big]{g|g(b)=0}&=1-f(b).
\end{align}
Since the sets $\mu\cbrace[\big]{g|g(b)=1}$ and
$\mu\cbrace[\big]{g|g(b)=0}$, $b=1,2,\dots,K$ are all the cylinder
sets of $\Gamma=\cbrace*{0,1}^K$, the measure $\mu_f$ is uniquely
determined as a product measure.

With this representation, it's a simple exercise to extend, say,
\cref{eq:98}, from indicators to nonnegative functions bounded by $1$.

There is only a slight complication when we have to extend results for
multiple averages from indicators to bounded functions.  For example,
for double averages, we need to be able to represent the product
$f_1(b_1)f_2(b_2)$ in an integral form.  This can be done using the
representation in \cref{eq:199} for $f_1$ and $f_2$ separately and
Fubini's theorem
\begin{align}
  \label{eq:202}
  f_1(b_1)f_2(b_2)&=\paren*{\int_{\mathcal{G}}g(b_1)d\mu_{f_1}(g)}\paren*{\int_{\mathcal{G}}g(b_2)d\mu_{f_2}(g)}
  \\
  &=\int_{\mathcal{G}\times \mathcal{G}}
  g_{1}(b_1)g_{2}(b_2)d\paren*{\mu_{f_1}\times\mu_{f_2}}(g_1,g_2),
\end{align}
which is just a fancy (but useful) way of writing
\begin{equation}
  \label{eq:203}
  f_1(b_1)f_2(b_2)=\mu_{f_1}\cbrace[\big]{g|g(b_1)=1}\cdot \mu_{f_2}\cbrace[\big]{g|g(b_2)=1}.
\end{equation}

\subsection{Extension to commuting transformations}
\label{sec:extens-comm-transf}

In short: our method in its present form doesn't work for commuting
transformations.  For example, we have
\begin{uproblem}
  Let $b<1/(\ell +1)$ and $\mathbbm u_n=n^{-b}$.

  Show that with probability $1$ in any dynamical system, for any
  commuting transformations $T_1,T_2,\dots,T_\ell$, the averages
  \begin{equation}
    \label{eq:204}
    \frac1{\sum_{n<N}\mathbbm u_n}\sum_{n<N} U_n\cdot
    T_1^nF_1\cdot T_2^{n}F_2\cdot\dots \cdot T_\ell^{n}F_\ell
  \end{equation}
  converge in $L^1$ norm.
\end{uproblem}
The good news is that the extension of \Cref{bigthm:3} to commuting
transformations is true. In other words, the semirandom averages
\begin{equation}
  \label{eq:205}
  \frac1N\sum_{n<N}T_1^nF_1(x)\cdot T_2^{r_n}F_2(x)
\end{equation}
do converge almost everywhere for commuting measure preserving
transformations $T_i$ as long as the random sequence
$\cbrace*{r_1<r_2<r_3<\dots}$ satisfies $r_n/n^{2-\epsilon}$ for some
positive $\epsilon$.  The proof of this will be discussed in an
upcoming paper.

\appendix
% \section{Appendix}
% \label{sec:appendix}

\section{Strong law of large numbers}
\label{sec:strong-law-large}
Here we give a quick proof of the strong law of large numbers in a
form needed in our paper.  The result is due to Kolmogorov.
\begin{thm}[Strong law of large numbers for uniformly bounded rv's]
  \label{thm:2}
  Suppose the random variables $X_1,X_2,\dots$ are non-negative,
  uniformly bounded
  \begin{equation}
    \label{eq:206}
    0\le X_n\le c,
  \end{equation}
  (pairwise) uncorrelated
  \begin{equation}
    \label{eq:207}
    \mathbbm E X_nX_m=\mathbbm E X_n\cdot \mathbbm E X_m, \text{ for
    }
    n\ne m,
  \end{equation}
  and dissipative
  \begin{equation}
    \label{eq:208}
    \sum_N \mathbbm E X_n=\infty.
  \end{equation}

  Then
  \begin{equation}
    \label{eq:209}
    \lim_{N\to\infty}\frac1{\sum_{n< N} \mathbbm EX_n}\sum_{n<N} X_n= 1 \text{ with
      probability } 1.
  \end{equation}

\end{thm}
For an example of two random variables which are nonnegative,
uncorrelated but not independent, take $X_1(x)=1+\sin(2\pi x)$ and
$X_2(x)=1+\sin(2\pi 2x)$ defined on the unit interval $[0,1]$.
(Looking at the graph of the two functions, by inspection we can find
a $\lambda$, $0<\lambda<2$, so that the level sets
$\cbrace*{X_1>\lambda}$ and $\cbrace*{X_2>\lambda}$ are nonempty, but
disjoint.)
\begin{proof}
  We want to use a subsequence argument. Define the set
  $I=\cbrace*{N_1<N_2<\dots<N_i<\dots}$ of indices by
  \begin{equation}
    \label{eq:210}
    \sum_{n<N_i} \mathbbm EX_n \ge i^2 \text{ but }
    \sum_{n<N} \mathbbm EX_n < i^2 \text{ for }N<N_i.
  \end{equation}
  In words: $N_i$ is the smallest index $N$ for which $\sum_{n<N}
  \mathbbm EX_n\ge i^2$.  That $I$ is well defined and infinite
  follows from the dissipative assumption $\sum_N \mathbbm
  EX_n=\infty$. Since the $X_n$ are uniformly bounded, we have
  \begin{equation}
    \label{eq:211}
    \sum_{n<N_i} \mathbbm EX_n=i^2+O(1).
  \end{equation}
  We first want to show that we have a.e. convergence along the
  subsequence $I$
  \begin{equation}
    \label{eq:212}
    \lim_{i\to\infty}\frac1{\sum_{n< N_i} \mathbbm EX_n}\sum_{n<N_i} X_n= 1 \text{ with
      probability } 1.
  \end{equation}
  We show this in the form
  \begin{equation}
    \label{eq:213}
    \sum_{i} \paren*{\frac1{\sum_{n<N_i} \mathbbm E X_{n}}\sum_{n<N_i}
      X_{n}-\mathbbm E X_{n}}^2<\infty \text{ with probability }1.
  \end{equation}
  This follows if the above series is integrable, so from
  \begin{equation}
    \label{eq:214}
    \mathbbm E\sum_{i} \paren*{\frac1{\sum_{n<N_i} \mathbbm E X_{n}}\sum_{n<N_i}
      X_{n}-\mathbbm E X_{n}}^2=\sum_{i}  \mathbbm E\paren*{\frac1{\sum_{n<N_i} \mathbbm E X_{n}}\sum_{n<N_i}
      X_{n}-\mathbbm E X_{n}}^2<\infty. 
  \end{equation}
  A consequence of the multiplicativity assumption in \cref{eq:207} is
  orthogonality
  \begin{equation}
    \mathbbm E(X_{n}-\mathbbm E X_{n})(X_{m}-\mathbbm E
    X_{m})=0\text{ for }n\ne m,\label{eq:215}  
  \end{equation}
  and hence we have
  \begin{equation}
    \label{eq:216}
    \mathbbm E\paren*{\sum_{n<N_i}
      X_{n}-\mathbbm E X_{n}}^2 =\sum_{n<N_i}\mathbbm E (X_{n}-\mathbbm E X_{n})^2.
  \end{equation}
  Since $\mathbbm E (X_{n}-\mathbbm E X_{n})^2\le \mathbbm E
  X_{n}^2\le c\mathbbm E X_{n}$, we have
  \begin{equation}
    \label{eq:217}
    \mathbbm E\paren*{\frac1{\sum_{n<N_i} \mathbbm E X_{n}}\sum_{n<N_i}
      X_{n}-\mathbbm E X_{n}}^2\le \frac {c}{\sum_{n<N_i} \mathbbm E X_{n}},
  \end{equation}
  which, by \cref{eq:211}, establishes \cref{eq:214}.

  We now want to show that the $N$th average is close to the $N_i$th
  if $N_i\le N<N_{i+1}$. We can estimate, since the $X_n$ are
  non-negative,
  \begin{align}
    \label{eq:218}
    \frac1{\sum_{n<N} \mathbbm EX_n}\sum_{n<N} X_n&\le
    \frac1{\sum_{n<N_i} \mathbbm EX_n}\sum_{n<N_{i+1} } X_n\\
    &=\frac{\sum_{n<N_{i+1} } \mathbbm EX_n}{\sum_{ n< N_{i} }\mathbbm
      EX_n}\cdot \frac1{\sum_{n<N_{i+1} } \mathbbm EX_n} \sum_{
      n<N_{i+1}} X_n.
  \end{align}
  This implies, by \cref{eq:212} and \cref{eq:211},
  \begin{align}
    \label{eq:219}
    \limsup_{N\to\infty}\frac1{\sum_{n<N} \mathbbm EX_n}\sum_{n <N}
    X_n&\le
    \lim_{i\to\infty}\frac{(i+1)^2+O(1)}{i^2+O(1)}\\
    &=1.\label{eq:220}
  \end{align}
  Similar estimate shows that
  \begin{equation}
    \label{eq:221}
    \liminf_{N\to\infty}\frac1{\sum_{n<N} \mathbbm EX_n}\sum_{n<N}
    X_n\ge1,
  \end{equation}
  finishing our proof.
\end{proof}

\bibliography{cikk}
\end{document}